\documentclass[reqno]{amsart}
\usepackage{amsfonts}
\usepackage{eurosym}
\usepackage{graphicx}
\usepackage{subfigure}
\usepackage{amssymb,amsmath,amsthm,amscd}

\newtheorem{theorem}{Theorem}[section]
\newtheorem{lemma}[theorem]{Lemma}
\newtheorem{proposition}[theorem]{Proposition}
\theoremstyle{definition}
\newtheorem{definition}[theorem]{Definition}

\theoremstyle{remark}
\newtheorem{remark}[theorem]{Remark}
\numberwithin{equation}{section}

\DeclareMathAlphabet\gothic{U}{euf}{m}{n}

\def\RR{\mathbb{R}}

\def\Om{\Omega}

\def\bOm{\overline{\Om}}

\begin{document}
\title[Degenerate parabolic equations]{On some degenerate non-local
parabolic equation associated with the fractional $p$-Laplacian}
\author{Ciprian G. Gal}
\address{C. G. Gal, Department of Mathematics, Florida International
University, Miami, 33199 (USA)}
\email{cgal@fiu.edu}
\author{Mahamadi Warma}
\address{M.~Warma, University of Puerto Rico, Faculty of Natural Sciences,
Department of Mathematics (Rio Piedras Campus), PO Box 70377 San Juan PR
00936-8377 (USA)}
\email{mahamadi.warma1@upr.edu, mjwarma@gmail.com}
\thanks{The work of the second author is partially supported by the Air
Force Office of Scientific Research under the Award No. FA9550-15-1-0027}
\subjclass[2010]{35R11, 35K55, 35K65}
\date{}
\keywords{Fractional $p$-Laplace operator, Dirichlet boundary conditions,
degenerate non-linear parabolic equations, existence and regularity of local
solutions, blow up.}

\begin{abstract}
Let $\Omega\subset\mathbb{R}^N$ be an arbitrary bounded open set. We
consider a degenerate parabolic equation associated to the fractional $p $%
-Laplace operator $\left( -\Delta \right) _{p}^{s}$\ ($p\geq 2$, $s\in
\left( 0,1\right) $) with the Dirichlet boundary condition and a monotone
perturbation growing like $\left\vert \tau\right\vert ^{q-2}\tau,$ $q>p$ and
with bad sign at infinity as $\left\vert \tau\right\vert \rightarrow \infty $%
. We show the existence of locally-defined strong solutions to the problem
with any initial condition $u_{0}\in L^{r}(\Omega )$ where $r\geq 2$
satisfies $r>N(q-p)/sp$. Then, we prove that finite time blow-up is possible
for these problems in the range of parameters provided for $r,p,q$ and the
initial datum $u_0$.
\end{abstract}

\maketitle
\tableofcontents

\section{Introduction}

The article is concerned with the following non-local initial-boundary value
problem for the degenerate parabolic equation
\begin{equation}
\begin{cases}
\displaystyle\partial _{t}u(x,t)+(-\Delta
)_{p}^{s}u(x,t)-|u(x,t)|^{q-2}u(x,t)=f(x,t) & (x,t)\in \Omega \times (0,T)
\\
u(x,t)=0 & (x,t)\in (\mathbb{R}^{N}\backslash \Omega )\times (0,T) \\
u(x,0)=u_{0}(x) & x\in \Omega .%
\end{cases}
\label{par-pro}
\end{equation}%
Here $u_{0}\in L^{r}(\Omega )$, $2\leq p,q,r<\infty $, $T>0$, $f$ is a given
function, $(-\Delta )_{p}^{s}$ denotes the fractional $p$-Laplace operator
and $\Omega $ is an arbitrary bounded open subset of $\mathbb{R}^{N}$, $%
N\geq 1$. To introduce the fractional $p$-Laplace operator, let $0<s<1$, $%
p\in (1,\infty )$ and set%
\begin{equation*}
\mathcal{L}^{p-1}(\mathbb{R}^{N}):=\left\{u:\;\mathbb{R}^{N}\rightarrow
\mathbb{R}\;\mbox{
measurable, }\;\int_{\mathbb{R}^{N}}\frac{|u(x)|^{p-1}}{(1+|x|)^{N+ps}}%
\;dx<\infty \right\}.
\end{equation*}%
For $u\in \mathcal{L}^{p-1}(\mathbb{R}^{N})$, $x\in \mathbb{R}^{N}$ and $%
\varepsilon >0$, we let
\begin{equation*}
(-\Delta )_{p,\varepsilon }^{s}u(x)=C_{N,p,s}\int_{\{y\in \mathbb{R}%
^{N},|y-x|>\varepsilon \}}|u(x)-u(y)|^{p-2}\frac{u(x)-u(y)}{|x-y|^{N+ps}}dy,
\end{equation*}%
where the normalized constant%
\begin{equation*}
C_{N,p,s}=\frac{s2^{2s}\Gamma\left(\frac{ps+p+N-2}{2}\right)}{\pi^{\frac
N2}\Gamma(1-s)}
\end{equation*}%
and $\Gamma $ is the usual Gamma function (see, e.g., \cite%
{BCF,Caf1,Caf2,Caf3,NPV} for the linear case $p=2,$ and \cite{War-IP,War-Pr}
for the general case $p\in (1,\infty )$). The fractional $p$-Laplacian $%
(-\Delta )_{p}^{s}$ is defined by the formula
\begin{align}
(-\Delta )_{p}^{s}u(x)=&C_{N,p,s}\mbox{P.V.}\int_{\mathbb{R}%
^{N}}|u(x)-u(y)|^{p-2}\frac{u(x)-u(y)}{|x-y|^{N+ps}}dy  \notag \\
=&\lim_{\varepsilon \downarrow 0}(-\Delta )_{p,\varepsilon
}^{s}u(x),\;\;\;x\in \mathbb{R}^{N},  \label{eq11}
\end{align}%
provided that the limit exists. We notice that if $0<s<\left( p-1\right) /p$
and $u$ is smooth (i.e., at least bounded and Lipschitz continuous), then
the integral in \eqref{eq11} is in fact not really singular near $x$.

The case $p=2$ and $f\equiv 0$, which corresponds to the case of a
semilinear fractional heat equation, sufficient conditions for the existence
of weak solutions with $u_{0}\in L^{2}\left( \Omega \right) ,$ and strong
solutions for $u_{0}\in L^{\infty }\left( \Omega \right) ,$ have already
been proved in \cite{GM}. Additionally, further dynamical properties (i.e.,
existence of finite dimensional global attractors and global asymptotic
stabilization to steady states as time goes to infinity) were also derived
for a semilinear parabolic problem of the form%
\begin{equation}
\partial _{t}u+(-\Delta )_{2}^{s}u+h\left( u\right) =0\;\;\mbox{ in }%
\;\Omega \times (0,\infty ),\text{ }u=0\;\mbox{ on }\;(\mathbb{R}%
^{N}\backslash \Omega )\times (0,\infty ),  \label{semi}
\end{equation}%
with nonlinearity $h\left( \tau\right) $ which has a good sign at infinity
as $\left\vert \tau\right\vert \rightarrow \infty $, and which is coercive
in a precise sense. Finally, some blow-up results were also proved in \cite%
{GM} for (\ref{semi}) with $h\left( \tau\right) \sim -\left\vert
\tau\right\vert ^{q-2}\tau $, as $\left\vert \tau\right\vert \rightarrow
\infty $, emphasizing the same critical blow-up exponent $q=p=2$ as for the
corresponding parabolic equation associated with the classical Laplace
operator $-\Delta $. We extend our work of \cite{GM} to prove the local in
time existence of solutions to parabolic equations with degenerate
fractional diffusion and more singular kernels using an approach based on
\cite{AO1, AO2}\ and also developed further in \cite{Aka}. Although our
general scheme follows closely that of \cite{Aka, AO1, AO2}, many of the key
lemmas used in the case of the classical $p$-Laplace operator cannot be
adapted or exploited in their classical form to deal with the fractional $p$%
-Laplacian $\left( -\Delta \right) _{p}^{s}$ for $s\in \left( 0,1\right) $
and $p\in \left( 1,\infty \right) $. Hence, we develop some new techniques
including some new functional inequalities allowing us to extend the results
of \cite{Aka} in the present setting. Among these new tools that we derive
it is worth mentioning a \emph{nonlinear version} of the classical
Stroock-Varopoulos inequality (see Lemma \ref{comp}) which is an important
inequality in the theory of Markovian semigroups, and a new coercitivity
estimate (see Lemma \ref{lem-9}) which is also crucial in the proofs for the
energy estimates. In particular, Lemma \ref{comp} extends the classical
Stroock-Varopoulos inequality which was available only in the case $p=2$
(see \cite{Str, Var}) and covers also the case when $p\neq 2$. Lemma \ref%
{comp} is the main tool in proving our first main result of Theorem \ref%
{main-theo}. Then we also generalize the blow-up results of \cite{GM} to the
present case (see Theorem \ref{theo2}) when $q>p$ following a technique
adapted from \cite{LW}. We emphasize that our results hold without any
regularity assumptions on $\Omega $. There is vast literature on degenerate
parabolic equations involving the classical diffusion operator $-\Delta _{p}$%
. We refer the reader to the following list \cite{Aka,AO1, AO2, Br73, LW}
(and references contained therein)\ which is not meant to be exhaustive.

To the best of our knowledge, little is known about parabolic problems
associated with the fractional $p$-Laplacian $(-\Delta )_{p}^{s}$ with the
exception of \cite{Tan, War-IP,War-Pr}. In \cite{War-Pr}, some regularity
results are provided for the quasi-linear parabolic equation $\partial
_{t}u+(-\Delta )_{p}^{s}u=0$ and Dirichlet boundary condition $u=0$ in $%
\mathbb{R}^{N}\backslash \Omega $, whereas in \cite{Tan} for the same
quasi-linear problem, it is proven the eventual boudedness of $u$ in $%
L^{\infty }\left( (\tau ,T\right) ;L^{\infty }\left( \Omega \right) ),$ for
every $\tau >0$, provided that the initial datum $u_{0}\in L^{p}\left(
\Omega \right) $. Most recently an integration by parts formula for the
regional fractional $p$-Laplace operator has been also derived in \cite%
{War-IP}.

\noindent \textit{Outline of paper.} In Section \ref{ss:fs}, we state the
relevant definitions and notation of fractional order Sobolev spaces.
Furthermore, in Section \ref{sec-main} we give a summary of the main results
but reserve the proofs for subsequent sections. In Section \ref{aux-pro}, we
introduce an auxiliary and a regularized version of the original problem and
prove some local existence results for them. Finally, the local existence
result for the original problem and then a finite time blow-up result are
proved in Section \ref{pro-main-theo}.

\section{Outline of results}

\subsection{Fractional order Sobolev spaces}

\label{ss:fs}

In this subsection, we recall some well-known results on fractional order
Sobolev spaces. To this end let $\Omega \subset \mathbb{R}^{N}$ be an
arbitrary open set with boundary $\partial \Omega $. For $p\in \lbrack
1,\infty )$ and $s\in (0,1)$, we denote by
\begin{equation*}
W^{s,p}(\Omega ):=\left\{ u\in L^{p}(\Omega ):\;\int_{\Omega }\int_{\Omega }%
\frac{|u(x)-u(y)|^{p}}{|x-y|^{N+ps}}dxdy<\infty \right\}
\end{equation*}%
the fractional order Sobolev space endowed with the norm
\begin{equation*}
\Vert u\Vert _{W^{s,p}(\Omega )}:=\left( \int_{\Omega }|u|^{p}\;dx+\frac{%
C_{N,p,s}}{2}\int_{\Omega }\int_{\Omega }\frac{|u(x)-u(y)|^{p}}{|x-y|^{N+ps}}%
dxdy\right) ^{\frac{1}{p}}.
\end{equation*}%
In order to handle a non-smooth $\Omega \subset \mathbb{R}^{N}$ in the case
when $\Omega $ is simply an open and bounded set, we let
\begin{equation*}
W_{0}^{s,p}(\Omega )=\overline{\mathcal{D}(\Omega )}^{W^{s,p}(\Omega )}\;\,%
\mbox{
and }\;\widetilde{W}^{s,p}(\Omega ):=\overline{W^{s,p}(\Omega )\cap C(%
\overline{\Omega })}^{W^{s,p}(\Omega )}.
\end{equation*}%
By definition, $W_{0}^{s,p}(\Omega )$ is the smallest closed subspace of $%
\widetilde{W}^{s,p}(\Omega )$ containing the space $\mathcal{D}(\Omega
):=C_{c}^{\infty }\left( \Omega \right) $ (equipped with the topology that
corresponds to convergence in the sense of test functions). If $p\in
(1,\infty )$, then one may characterize the space $W_{0}^{s,p}(\Omega )$ as
follows (considering $W_{0}^{s,p}(\Omega )$ as a subspace of $\widetilde{W}%
^{s,p}(\Omega )$)%
\begin{equation*}
W_{0}^{s,p}(\Omega )=\{u\in \widetilde{W}^{s,p}(\Omega ):\;\tilde{u}=0\;%
\mbox{
quasi-everywhere on }\;\partial \Omega \},
\end{equation*}%
where $\tilde{u}$ is the quasi-continuous version of $u$ with respect to the
capacity defined with the space $\widetilde{W}^{s,p}(\Omega )$ (cf. \cite[%
Theorem 4.5]{War-N}). Finally we define the space
\begin{equation*}
W_{0}^{s,p}(\overline{\Omega })=\{u\in W^{s,p}(\mathbb{R}^{N}):\;u=0\;%
\mbox{
a.e. on }\;\mathbb{R}^{N}\setminus \Omega \}.
\end{equation*}%
It is clear that $W_{0}^{s,p}(\Omega )$ and $W_{0}^{s,p}(\overline{\Omega })$
are both subspace of $W^{s,p}(\Omega )$, but there is no obvious inclusion
between $W_{0}^{s,p}(\Omega )$ and $W_{0}^{s,p}(\overline{\Omega })$. We
notice that $W_{0}^{s,p}(\overline{\Omega })$ contains the space of test
functions $\mathcal{D}(\Omega )$ but the latter space is not always dense in
$W_{0}^{s,p}(\overline{\Omega })$. It has been proved in \cite[Theorem
1.4.2.2]{Gris} (see also \cite{FV} for some more general spaces) that if $%
\Omega $ has a continuous boundary, then $\mathcal{D}(\Omega )$ is dense in $%
W_{0}^{s,p}(\overline{\Omega })$. In addition if $\Omega $ has a Lipschitz
continuous boundary and $s\neq 1/p$, then $W_{0}^{s,p}(\Omega )=W_{0}^{s,p}(%
\overline{\Omega })$ with equivalent norm.

Throughout the remainder of the paper, we make the convention that if we
write $u\in W_{0}^{s,p}(\overline{\Omega} )$ we mean that $u\in W^{s,p}(%
\mathbb{R}^{N})$ and $u=0$ a.e. on $\mathbb{R}^{N}\backslash \Omega $. In
that sense, a simple calculation shows that
\begin{equation}
\Vert |u\Vert |_{W_{0}^{s,p}(\overline{\Omega})}=\left( \frac{C_{N,p,s}}{2}%
\int_{{\mathbb{R}}^{N}}\int_{{\mathbb{R}}^{N}}\frac{|u(x)-u(y)|^{p}}{%
|x-y|^{N+sp}}dxdy\right) ^{\frac{1}{p}}  \label{norm-26}
\end{equation}%
defines an equivalent norm on the space $W_{0}^{s,p}(\overline{\Omega})$. We
shall always use this norm for the space $W_{0}^{s,p}(\overline{\Omega})$
even when $\Omega $ is simply an open bounded subset of $\mathbb{R}^{N}$.
Let $p^{\star }$ be given by
\begin{equation}
p^{\star }=\frac{Np}{N-sp}\;\mbox{ if }\;N>sp\;\mbox{ and }\;p^{\star }\in
\lbrack p,\infty )\;\mbox{ if }\;N=sp.  \label{p-star}
\end{equation}
Then by \cite[Section 7]{NPV}, there exists a constant $C=C\left(
N,p,s\right) >0$ such that for every $u\in W_{0}^{s,p}(\overline{\Omega})$,
\begin{equation}
\Vert u\Vert _{q,\Omega }\leq C\Vert u\Vert _{W_{0}^{s,p}(\overline{\Omega}
)},\;\;\forall \;q\in \lbrack p,p^{\star }].  \label{sob-emb}
\end{equation}%
Since $\Omega $ is bounded, we have that \eqref{sob-emb} also holds for
every $q\in \lbrack 1,p^{\star }]$. Moreover, the embedding $W_{0}^{s,p}(%
\overline{\Omega})\hookrightarrow L^{q}(\Omega )$ is compact for every $q\in
\lbrack 1,p^{\star })$. The following version of the Gagliardo-Nirenberg
inequality for the space $W_{0}^{s,p}(\overline{\Omega})$ in the non-smooth
setting will be used. Let $p\in (1,\infty )$, $q,r\in \lbrack 1,\infty ]$
and $0\leq \alpha \leq 1$ satisfy
\begin{equation}
\frac{1}{q}=\frac{\alpha }{p^{\star }}+\frac{1-\alpha }{r}=\frac{N-sp}{Np}%
\alpha +\frac{1-\alpha }{r}.  \label{cond-Gag-Nir}
\end{equation}
Then there exists a constant $C>0$ such that for every $u\in W_{0}^{s,p}(%
\overline{\Omega})$,
\begin{equation}
\Vert u\Vert _{L^{q}(\Omega )}\leq \Vert u\Vert _{L^{p^{\star }}(\Omega
)}^{\alpha }\Vert u\Vert _{L^{r}(\Omega )}^{1-\alpha } \leq C\Vert |u\Vert
|_{W_{0}^{s,p}(\overline{\Omega} )}^{\alpha }\Vert u\Vert _{L^{r}(\Omega
)}^{1-\alpha }.  \label{Gag-Nir-2}
\end{equation}%
If $0<s<1$, $p\in (1,\infty )$ and $p^{\prime }=p/(p-1)$, the space $%
W^{-s,p^{\prime }}(\Omega )$ is defined as usual to be the dual of the
reflexive Banach space $W_{0}^{s,p}(\overline{\Omega} )$, that is, $%
(W_{0}^{s,p}(\overline{\Omega} ))^{\star }=W^{-s,p^{\prime }}(\Omega )$. For
more information on fractional order Sobolev spaces we refer the reader to
\cite{AH,NPV,Gris,JW,LM1,War-N} and the references contained therein.

\subsection{Main results}

\label{sec-main}

Let $\Omega \subset \mathbb{R}^{N}$ be an arbitrary bounded open set. As
usual for a Banach space $X,$ we denote by $C_{w}([a,b];X)$ the set of all $%
X $-valued weakly continuous functions on the interval $[a,b]$. We also
denote by $\left\langle \cdot ,\cdot \right\rangle _{X^{\ast },X}$ the
duality between $X^{\ast }$ and $X$. First, we introduce the rigorous notion
of solution to the system \eqref{par-pro}.

\begin{definition}
Let $0<s<1$, $2\leq p,q,r<\infty $, $p^{\prime }=p/(p-1)$, $r^{\prime
}=r/(r-1)$ and $q^{\prime }=q/(q-1)$. Let $u_{0}\in L^{r}(\Omega )$ and%
\begin{equation*}
f\in W^{1,p^{\prime }}((0,T);W^{-s,p^{\prime }}(\Omega )+L^{r^{\prime
}}(\Omega ))\cap L^{1+\gamma }((0,T);L^{r}(\Omega ))
\end{equation*}%
for some $\gamma \geq 0$ and $T>0$. A function $u$ is said to be a (strong)
solution of \eqref{par-pro} if%
\begin{equation}
\begin{cases}
\displaystyle u\in L^{\infty }((0,T);L^{r}(\Omega ))\cap
L^{p}((0,T);W_{0}^{s,p}(\overline{\Omega }))\cap L^{q}((0,T);L^{q}(\Omega )),
\\
\displaystyle\partial _{t}u\in L^{q^{\prime }}((0,T);W^{-s,p^{\prime
}}(\Omega )+L^{r^{\prime }}(\Omega )), \\
\displaystyle u(t)\in W_{0}^{s,p}(\overline{\Omega })\cap L^{r}(\Omega ),%
\text{ a.e. }t\in (0,T), \\
\displaystyle u\in W_{loc}^{1,2}((0,T);L^{2}(\Omega )),%
\end{cases}%
\end{equation}%
and, a.e. $t\in \left( 0,T\right) $ for every $v\in W_{0}^{s,p}(\overline{%
\Omega })\cap L^{r}(\Omega )=:V$, with $r>\frac{N(q-p)}{sp},$%
\begin{align}
& \left\langle \partial _{t}u\left( t\right) ,v\right\rangle _{V^{\ast },V}
\label{idd} \\
+& \frac{C_{N,p,s}}{2}\int_{\mathbb{R}^{N}}\int_{\mathbb{R}%
^{N}}|u(x,t)-u(y,t)|^{p-2}\frac{(u(x,t)-u(y,t))(v(x)-v\left( y\right) )}{%
|x-y|^{N+sp}}dxdy  \notag \\
=& \left\langle |u\left( t\right) |^{q-2}u\left( t\right) ,v\right\rangle
_{V^{\ast },V}+\left\langle f\left( t\right) ,v\right\rangle _{V^{\ast },V},
\notag
\end{align}%
and $u$ satisfies the initial condition
\begin{equation*}
u(\cdot ,t)\rightarrow u_{0}\;\mbox{ strongly in }\;L^{r}(\Omega )\;%
\mbox{
as }\,t\rightarrow 0^{+}.
\end{equation*}
\end{definition}

\begin{remark}
Notice that $\left\langle |u\left( t\right) |^{q-2}u\left( t\right)
,v\right\rangle _{V^{\ast },V}$ on the right-hand side of (\ref{idd}) is
well-defined since for $r>\frac{N(q-p)}{sp}$, $V\subset L^{q}\left( \Omega
\right) $ boundedly (see also (\ref{cond-7}) below).
\end{remark}

The following is the first main result of the article.

\begin{theorem}
\label{main-theo} Let $T>0$ be fixed, $0<s<1$ and $p,q,r\in \lbrack 2,\infty
)$ be such that $p<q$ and assume that
\begin{equation*}
r>\frac{N(q-p)}{sp}.
\end{equation*}%
Let $u_{0}\in L^{r}(\Omega )$ and assume%
\begin{equation*}
f\in W^{1,p^{\prime }}((0,T);W^{-s,p^{\prime }}(\Omega )+L^{r^{\prime
}}(\Omega ))\cap L^{1+\gamma }((0,T);L^{r}(\Omega ))
\end{equation*}%
for some $\gamma \geq 0$. Then the following assertions hold.

\begin{enumerate}
\item[(a)] If $\gamma >0$, then there exist a non-increasing function $%
T_{\star }:\;[0,\infty )\times \lbrack 0,\infty )\rightarrow (0,T]$
independent of $T,u_{0}$ and $f$ and%
\begin{equation*}
T_{0}:=T_{\star }\left(\Vert u_{0}\Vert _{L^{r}(\Omega )},\int_{0}^{T}\Vert
f(t)\Vert _{L^{r}(\Omega )}^{1+\gamma }\;dt\right),
\end{equation*}%
such that \eqref{par-pro} has at least one strong solution on $(0,T_{0})$.

\item[(b)] If $\gamma =0$, then there exist a non-increasing function $%
T_{f}:\;[0,\infty )\rightarrow (0,T]$ independent of $T$ and $u_{0}$, and $%
T_{0}:=T_{f}(\Vert u_{0}\Vert _{L^{r}(\Omega )})$, such that \eqref{par-pro}
has at least one strong solution on $(0,T_{0})$.

\item[(c)] The strong solution has in addition the following regularity:%
\begin{equation*}
\left\{
\begin{array}{l}
|u|^{\frac{r-2}{p}}u\in L^{p}((0,T_{0});W_{0}^{s,p}(\overline{\Omega} )),%
\text{ }(-\Delta )_{p}^{s}u\in L^{p^{\prime }}((0,T_{0});W^{-s,p^{\prime
}}(\Omega )), \\
t^{\frac{1}{p}}u\in C_{w}([0,T_{0}];W_{0}^{s,p}(\overline{\Omega} )),\text{ }%
\sqrt{t}\partial _{t}u\in L^{2}((0,T_{0});L^{2}(\Omega )).%
\end{array}%
\right.
\end{equation*}
\end{enumerate}
\end{theorem}

The proof of Theorem \ref{main-theo} relies on rewriting \eqref{par-pro} as
a first order Cauchy problem which is governed by the difference of two
subdifferential operators in reflexive Banach spaces following the work of
\cite{Aka}. A family of approximate problems and refined energy estimates
will be employed to construct solutions with initial data $u_{0}\in
L^{r}\left( \Omega \right) $. The primary new difficulty, due to the
nonlocal character of the \emph{fractional} $p$-Laplacian, is obtaining a
new comparison lemma for various energy forms (see Lemma \ref{comp}) and
several other critical lemmas properly modified from \cite{Aka}\ to handle
our case. Solutions are first constructed for some auxiliary problems
associated with \eqref{par-pro}.

The second main result deals with blow-up phenomena for the strong solutions
of (\ref{par-pro}). To this end, we define the following energy functional%
\begin{equation}
E\left( t\right) :=\frac{C_{N,p,s}}{2p}\int_{\mathbb{R}^{N}}\int_{\mathbb{R}%
^{N}}\frac{|u(x,t)-u(y,t)|^{p}}{|x-y|^{N+sp}}dxdy-\frac{1}{q}\int_{\Omega
}|u\left( x,t\right) |^{q}dx  \label{energy}
\end{equation}%
and notice that when $f\equiv 0,$%
\begin{equation}
\frac{d}{dt}E\left( t\right) =-\left\Vert \partial _{t}u\left( t\right)
\right\Vert _{L^{2}\left( \Omega \right) }^{2}\leq 0  \label{en-ineq-smooth}
\end{equation}%
for as long as a \emph{smooth} solution exists. In fact, every strong
solution of Theorem \ref{main-theo} satisfies an energy inequality, as
follows.

\begin{proposition}
\label{en-in-prop}Let $u$ be a solution in the sense of Theorem \ref%
{main-theo} and further assume that $u_{0}\in W_{0}^{s,p}(\overline{\Omega})$
and $f\equiv 0$. Then%
\begin{equation}
E\left( t\right) \leq E\left( 0\right) ,  \label{en-ineq}
\end{equation}%
for almost all $t\in \left( 0,T_{0}\right) $, for as long as a strong
solution exists.
\end{proposition}

\begin{theorem}
\label{theo2}Let $u$ be a strong solution of \eqref{par-pro} in the sense of
Theorem \ref{main-theo} and $f\equiv 0$. Let $u_{0}\in W_{0}^{s,p}(\overline{%
\Omega} )$ such that $E\left( 0\right) <E_{0}$ and $\left\Vert \left\vert
u_{0}\right\vert \right\Vert _{W_{0}^{s,p}(\overline{\Omega} )}>\alpha $ with%
\begin{equation*}
\alpha =C_{\ast }^{-\frac{q}{q-p}}\text{, }E_{0}=\left( \frac{1}{p}-\frac{1}{%
q}\right) C_{\ast }^{-\frac{qp}{q-p}},
\end{equation*}%
where $C_{\ast }>0$ is the best Sobolev constant in \eqref{sob-emb} and $%
q\in (p,p^{\ast }].$ Then the strong solution blows-up in a finite time $%
t_{\ast }>0$ with
\begin{equation}
t_{\ast }\leq \frac{\left( \frac{1}{2}\right) ^{q-1}\left\Vert
u_{0}\right\Vert _{L^{2}\left( \Omega \right) }^{q-2}\left\vert \Omega
\right\vert ^{\frac q2-1}}{\left( \frac{q}{2}-1\right) \left( 1-\frac{\alpha
^{q}}{\beta ^{q}}\right) \left( q-p\right) },  \label{bl-time}
\end{equation}%
for some $\beta >\alpha .$
\end{theorem}

\begin{remark}
These results can be also extended to degenerate parabolic equations of the
form%
\begin{equation*}
\displaystyle\partial _{t}u+L_{p,s,\Omega }\left( u\right) -g\left( u\right)
=f\left( x,t\right) ,\text{ }(x,t)\in \Omega \times (0,T),
\end{equation*}%
subject to the condition $u=0$ in $\mathbb{R}^{N}\backslash \Omega $, where%
\begin{equation*}
L_{p,s,\Omega }\left( u\left( x\right) \right) :=\mbox{P.V.}\int_{\mathbb{R}%
^{N}}\gothic{a}\left( u\left( x\right) ,u\left( y\right) \right) \left(
\frac{u(x)-u(y)}{|x-y|^{N+ps}}\right) dy
\end{equation*}%
with $\gothic{a}\in C\left( \mathbb{R}^2,\mathbb{R}_{+}\right) $ satisfying
the following condition%
\begin{equation*}
c_{p}\left\vert \tau_{1}-\tau_{2}\right\vert ^{p-2}\leq \gothic{a}\left(
\tau_{1},\tau_{2}\right) \leq c_{0}(1+\left\vert
\tau_{1}-\tau_{2}\right\vert ^{p-2})\text{,}
\end{equation*}%
for all $\tau_{1},\tau_{2}\in \mathbb{R}$, for some $c_{0},c_{p}>0.$ The
function $g$ is a maximal monotone graph in $\mathbb{R}^{2}$ such that $%
\left\vert g\left( s\right) \right\vert \leq c_{g}\left\vert s\right\vert
^{q-1}$ as $\left\vert s\right\vert \rightarrow \infty $. We leave the
details to the interested reader.
\end{remark}

\section{Auxiliary and regularized problems}

\label{aux-pro}

\subsection{Subdifferentials}

In this subsection we introduce some useful properties of subdifferentials
of proper, convex and lower semi-continuous functionals on a Banach space.

\begin{definition}
Let $X$ be a reflexive Banach space.

\begin{enumerate}
\item A mapping $\varphi:X\to(-\infty,\infty]$ is called \emph{proper} if
its effective domain
\begin{align*}
D(\varphi):=\{x\in X:\varphi(x)<\infty\}
\end{align*}
is not empty. For a proper mapping $\varphi:X\to(-\infty,\infty]$, we define
the \emph{convex conjugate} $\varphi^\star$ by
\begin{align*}
\varphi^\star:X^\star\to(-\infty,\infty],\quad
\varphi^\star(x^\star):=\sup_{x\in X} x^\star(x)-\varphi(x).
\end{align*}
Note that $\varphi^\star$ is convex even if $\varphi$ is not.

\item Given a mapping $\varphi:X\to(-\infty,\infty]$ and $x_0\in X$, a
functional $x^\star\in X^\star$ is called a \emph{subgradient} of $\varphi$
at $x_0$ if for all $x\in X$, we have
\begin{align*}
x^\star(x-x_0) \leq \varphi(x)-\varphi(x_0).
\end{align*}
The set of all these subgradients is called the \emph{subdifferential} of $%
\varphi$ at $x_0$ and is denoted by $\partial_X\varphi(x_0)$. The domain $%
D(\partial_X\varphi)$ of the subdifferential $\partial_X\varphi$ is given by
\begin{align*}
D(\partial_X\varphi):=\{x\in X:\; \partial_X\varphi(x)\ne \emptyset\}.
\end{align*}
Obviously, $D(\partial_X\varphi)\subset D(\varphi)$.
\end{enumerate}
\end{definition}

It is well-known (see e.g. \cite{Br73,Scho}) that every subdifferential of a
proper, convex and lower semi-continuous functional is maximal monotone.
Moreover, if $X=H$ is a Hilbert space then the subdifferential $\partial
_{H}\varphi $ can be written for $u\in D(\varphi)$ as
\begin{equation*}
\partial _{H}\varphi (u)=\{w\in H:\;\varphi (v)-\varphi (u)\geq
(w,v-u)_{H},\;\text{for all }v\in D(\varphi )\},
\end{equation*}%
where $(\cdot ,\cdot )_{H}$ denotes the inner product of $H$, and also $%
\partial _{H}\varphi $ becomes a maximal monotone operator on $H$. For a
proper, convex and lower semi-continuous functional $\varphi $ on $H$, the
\emph{Moreau-Yosida approximation} $\varphi _{\lambda }$ of $\varphi $ is
defined as follows:
\begin{equation}
\varphi _{\lambda }(u):=\inf_{v\in H}\left\{ \frac{1}{2\lambda }\Vert
u-v\Vert _{H}^{2}+\varphi (v)\right\} ,\;\;\text{for all}\;u\in H,\;\lambda
>0.  \label{MY-Apr}
\end{equation}%
We recall that the \emph{Yosida approximation} of a maximal monotone
operator $A$ on a Hilbert space $H$ is defined as%
\begin{equation}
A_{\lambda }:=\frac{1}{\lambda }\left[ I-\left( I+\lambda A\right) ^{-1}%
\right] ,\;\;\lambda >0.  \label{Y-Apr}
\end{equation}

The following result provides some useful properties of Moreau-Yosida and
Yosida approximations. Its proof can be found in \cite[Proposition 2.11, p.39%
]{Br73}.

\begin{proposition}
\label{more} Let $\varphi $ be a proper, convex and lower semi-continuous
functional on $H$ and $\varphi _{\lambda }$ be its Moreau-Yosida
approximation. Then $\varphi _{\lambda }$ is convex, Fr\'{e}chet
differentiable in $H$, and its Fr\'echet derivative $\partial _{H}(\varphi
_{\lambda }) $ coincides with the Yosida approximation $(\partial _{H}\varphi
)_{\lambda }$ of $\partial _{H}\varphi $. Moreover, the following properties
hold:
\begin{equation}
\begin{cases}
\varphi _{\lambda }(u)=\frac{1}{2\lambda }\Vert u-J_{\lambda }^{\varphi
}u\Vert _{H}^{2}+\varphi (J_{\lambda }^{\varphi }u),\; & \text{ for all }%
u\in H,\lambda >0, \\
\varphi (J_{\lambda }^{\varphi }u)\leq \varphi _{\lambda }(u)\leq \varphi
(u),\; & \text{ for all }u\in H,\lambda >0, \\
\varphi (J_{\lambda }^{\varphi }u)\uparrow \varphi (u)\;\mbox{ as }\;\lambda
\rightarrow 0^{+},\;\;\; & \text{ for all }u\in H,%
\end{cases}
\label{eq-pro-more}
\end{equation}%
where $J_{\lambda }^{\varphi }:=(I+\lambda \partial _{H}\varphi )^{-1}$ is
the resolvent operator of $\partial _{H}\varphi $.
\end{proposition}

The following type of chain rule for subdifferentials is taken from \cite[%
Proposition 5]{Aka}.

\begin{proposition}
\label{prop-chain} Let $X$ be a reflexive Banach space, $T>0$ be fixed and
let $\varphi :X\rightarrow (-\infty ,\infty ]$ be a proper, convex and lower
semi-continuous functional. Let $p\in (1,\infty )$ and let $u\in
W^{1,p}((0,T);X)$ be such that $u(t)\in D(\partial _{X}\varphi )$ for a.e. $%
t\in (0,T)$. Suppose that there exits $g\in L^{p^{\prime }}((0,T);X^{\star
}) $ such that $g(t)\in \partial _{X}\varphi (u(t))$ for a.e. $t\in (0,T)$.
Then the function $t\mapsto \varphi (u(t))$ is differentiable for a.e. $t\in
(0,T) $. Moreover, for a.e. $t\in (0,T)$,
\begin{equation}
\frac{d}{dt}\varphi (u(t))=\left\langle f,\frac{du}{dt}(t)\right\rangle_{X^%
\star,X} \;\mbox{ for all }\;f\in \partial _{X}\varphi (u(t)).
\label{eq-chain}
\end{equation}
\end{proposition}

Next, let $\theta $ be a maximal monotone graph in $\mathbb{R}^{2}$. In the
following result, for a given $u\in L^{2}(\Omega )$ we discuss the
representation of $\theta (u(\cdot ))$ as the subdifferential $\partial
_{L^{2}(\Omega )}\Theta (u)$ for some proper, convex and lower
semi-continuous functional $\Theta $ on $L^{2}(\Omega )$.

\begin{proposition}
\label{prop-24} Let $\Omega \subset \mathbb{R}^{N}$ be an open and bounded
set and let $\theta :\mathbb{R}\rightarrow ( -\infty ,\infty ]$ be a proper,
convex and lower semi-continuous functional. Define the functional $\Theta
:L^{2}(\Omega )\rightarrow (-\infty ,\infty ]$ with effective domain $%
D(\Theta )=\{u\in L^{2}(\Omega ):\;\theta (u(\cdot ))\in L^{1}(\Omega )\}$
and given by
\begin{equation*}
\Theta (u):=%
\begin{cases}
\displaystyle\int_{\Omega }\theta (u(x))dx\;\; & \mbox{ if }\;u\in D(\Theta
), \\
+\infty & \mbox{ otherwise}.%
\end{cases}%
\end{equation*}%
Let $J_{\lambda }^{\Theta }$ and $j_{\lambda }^{\theta }$ ($\lambda >0$)
denote the resolvent operators of the subdifferentials $\partial
_{L^{2}(\Omega )}\Theta $ and $\partial _{\mathbb{R}}\theta $, respectively.
Then the following properties hold.

\begin{enumerate}
\item The functional $\Theta $ is proper, convex and lower semi-continuous
on $L^{2}\left( \Omega \right) $.

\item For all $f,u\in L^2(\Omega)$, we have that $f\in
\partial_{L^2(\Omega)}\Theta(u)$ if and only if $f(x)\in \partial_{\mathbb{R}%
}\theta(u(x))$ for a.e. $x\in\Omega$.

\item For all $u\in L^2(\Omega)$, $J_\lambda^\Theta u(x)=j_\lambda^\theta
u(x)$ for a.e. $x\in\Omega$ and for all $\lambda>0$.

\item For every $m\in \lbrack 1,\infty ]$, if $u,v\in L^{m}(\Omega )\cap
L^{2}(\Omega )$, then $J_{\lambda }^{\Theta }u,\partial _{L^{2}(\Omega
)}\Theta _{\lambda }(u)\in L^{m}(\Omega )\cap L^{2}(\Omega )$ for all $%
\lambda >0$ and
\begin{align*}
\Vert J_{\lambda }^{\Theta }u-J_{\lambda }^{\Theta }v\Vert _{L^{m}(\Omega
)}\leq \Vert u-v\Vert _{L^{m}(\Omega )}, \\
\Vert \partial _{L^{2}(\Omega )}\Theta _{\lambda }(u)-\partial
_{L^{2}(\Omega )}\Theta _{\lambda }(v)\Vert _{L^{m}(\Omega )}\leq \frac{2}{%
\lambda }\Vert u-v\Vert _{L^{m}(\Omega )}.
\end{align*}


\item If $\partial _{\mathbb{R}}\theta (0)\ni 0$, then for every $p\in
(1,\infty )$ and $s\in (0,1)$, we have that $J_{\lambda }^{\Theta }0=0$, $%
J_{\lambda }^{\Theta }u\in W_{0}^{s,p}(\overline{\Omega} )\cap L^{2}(\Omega
) $ for all $u\in W_{0}^{s,p}(\overline{\Omega} )\cap L^{2}(\Omega )$ and
for all $\lambda >0$. Moreover,
\begin{equation}
\int_{\mathbb{R}^{N}}\int_{\mathbb{R}^{N}}\frac{|J_{\lambda }^{\Theta
}u(x)-J_{\lambda }^{\Theta }u(y)|^{p}}{|x-y|^{N+sp}}dxdy\leq \int_{\mathbb{R}%
^{N}}\int_{\mathbb{R}^{N}}\frac{|u(x)-u(y)|^{p}}{|x-y|^{N+sp}}dxdy.
\label{Ine-2}
\end{equation}
\end{enumerate}
\end{proposition}

\begin{proof}
The proof of parts (a), (b) (c) and (d) is contained in \cite[Proposition 6]%
{Aka} (see also \cite[Proposition 8.1]{Scho} for parts (a) and (b) and also
\cite[Proposition 2.16, p.47]{Br73}).

Next, let $\lambda >0$, $p\in (1,\infty )$, $s\in (0,1)$ and $u\in
W^{s,p}(\RR^N )\cap L^{2}(\RR^N )$. It follows from part (d) that $%
J_{\lambda }^{\Theta }u\in L^{p}(\RR^N)\cap L^{2}(\RR^N )$. Since
\begin{equation*}
|j_{\lambda }^{\theta }(u(x))-j_{\lambda }^{\theta }(u(y))|\leq |u(x)-u(y)|\;%
\mbox{ for
a.e. }\;x,y\in \RR^N ,
\end{equation*}%
then we obtain \eqref{Ine-2} by using the assertion (c).

It remains to show the assertion (e). First, let $\lambda >0$, $p\in
(1,\infty )$, $s\in (0,1)$ and $u\in W^{s,p}(\mathbb{R}^N)\cap L^{2}(\Omega
) $. It follows from part (d) that $J_{\lambda }^{\Theta }u\in L^{p}(\Omega
)\cap L^{2}(\Omega )$. Since
\begin{equation}\label{In-eq}
|j_{\lambda }^{\theta }(u(x))-j_{\lambda }^{\theta }(u(y))|\leq |u(x)-u(y)|\;%
\mbox{ for
a.e. }\;x,y\in \RR^N,
\end{equation}%
then we obtain \eqref{Ine-2} by using the assertion (c). Next, assume that $%
\partial _{\mathbb{R}}\theta (0)\ni 0$. Then it is clear that $j_{\lambda
}^{\theta }0=0$ and hence, $J_{\lambda }^{\Theta }0=0$ and $|J_{\lambda
}^{\Theta }u(x)|\leq |u(x)|$ for a.e. $x\in \Omega $. Let $u\in
W_{0}^{s,p}(\bOm )\cap L^{2}(\Omega )\subset W^{s,p}(\RR^N )\cap L^{2}(\Omega )$ and
$\lambda >0$. Since $J_{\lambda}^{\Theta }u\in L^p(\Omega)\cap L^2(\Omega)$ (see above), it follows from \eqref{In-eq} and part (c) that $J_{\lambda}^{\Theta }u\in W^{s,p}(\mathbb{R}^N)\cap L^{2}(\Omega)$. Since $|J_{\lambda
}^{\Theta }u(x)|\leq |u(x)|$ for a.e. $x\in \RR^N $ we also have that $J_{\lambda}^{\Theta }u=0$ on $\RR^N\setminus\Omega$.
Therefore $J_{\lambda
}^{\Theta }u\in W_{0}^{s,p}(\bOm )\cap L^{2}(\Omega )$ and we have shown part (e). The proof of
the proposition is finished.
\end{proof}

\subsection{The auxiliary problems}

We first write the system \eqref{par-pro} as a first order Cauchy problem.
To this end recall that $0<s<1$, $p,r\in \lbrack 2,\infty )$ and denote $%
V:=W_{0}^{s,p}(\overline{\Omega} )\cap L^{r}(\Omega )$ as the Banach space
equipped with the norm
\begin{equation*}
\Vert u\Vert _{V}:=\left( \Vert u\Vert _{L^{r}(\Omega )}^{2}+\Vert |u\Vert
|_{W_{0}^{s,p}(\overline{\Omega} )}^{2}\right) ^{\frac{1}{2}}.
\end{equation*}%
where the second norm is given by \eqref{norm-26}. Let $V^{\star }$ denote
the dual of the reflexive Banach space $V$. Then
\begin{equation*}
V^{\star }=W^{-s,p^{\prime }}(\Omega )+L^{r^{\prime }}(\Omega
):=\{u=u_{1}+u_{2};\;u_{1}\in W^{-s,p^{\prime }}(\Omega ),\;u_{2}\in
L^{r^{\prime }}(\Omega )\},
\end{equation*}%
where $p^{\prime }=p/(p-1)$ and $r^{\prime }=r/(r-1)$. For every $r\geq 2$,
we have the continuous injections $V\hookrightarrow L^{2}(\Omega
)\hookrightarrow V^{\star }$.\newline

Next, for $p,q,r\in \lbrack 2,\infty )$ satisfying $p<q$ and%
\begin{equation}
r>\frac{N(q-p)}{sp}\;\Longleftrightarrow \;q<\frac{N+sr}{N}p,  \label{cond-7}
\end{equation}%
we have that $V$ is continuously embedded into $L^{q}(\Omega )$. Indeed, if $%
q\leq r$, then $V$ is trivially continuously embedded into $L^{q}(\Omega )$.
If $r\leq q$, then it follows from \eqref{cond-7} that $q<p^{\star }$. Since
$W_{0}^{s,p}(\overline{\Omega})\hookrightarrow L^{p^{\star }}(\Omega )$, we
also have $V\hookrightarrow L^{q}(\Omega )$.

Let us now define the functionals $\Phi ,\psi :\;V\rightarrow \lbrack
0,\infty )$ by
\begin{equation*}
\Phi (u):=\frac{C_{N,p,s}}{2p}\int_{\mathbb{R}^{N}}\int_{\mathbb{R}^{N}}%
\frac{|u(x)-u(y)|^{p}}{|x-y|^{N+sp}}dxdy,\;\;\;\psi (u):=\frac{1}{q}%
\int_{\Omega }|u|^{q}dx,
\end{equation*}%
for all$\;u\in V$. It is easy to see that $\Phi ,\psi \in C^{1}(V,\mathbb{R}%
) $. We state the following basic proposition whose proof is postponed until
the Appendix.

\begin{proposition}
\label{subdif}Let $\partial _{V}\Phi $ and $\partial _{V}\psi $ denote the
single valued subgradients of $\Phi $ and $\psi $, respectively. Then $%
\partial _{V}\Phi $ is an operator from $V$ to $V^{\star }$ and can be
expressed as
\begin{equation*}
D(\partial _{V}\Phi )=V\;\;\mbox{ and }\;\partial _{V}\Phi (u)=(-\Delta
)_{p}^{s}u,\;\;\text{for all}\;u\in V.
\end{equation*}%
More precisely, $\partial _{V}\Phi $ is a realization in $V^{\star }$ of the
fractional $p$-Laplace operator $(-\Delta )_{p}^{s}$ with the Dirichlet
boundary condition $u=0$ on $\mathbb{R}^{N}\backslash \Omega $. Finally,
under the assumption \eqref{cond-7}, we also have that $\partial _{V}\psi $
is an operator from $V$ to $V^{\star }$ with
\begin{equation*}
D(\partial _{V}\psi )=V\;\;\mbox{ and }\;\partial _{V}\psi
(u)=|u|^{q-2}u,\;\;\text{for all}\;u\in V.
\end{equation*}
\end{proposition}

By virtue of Proposition \ref{subdif}, the system \eqref{par-pro} can be
rewritten as the following abstract Cauchy problem%
\begin{equation}
\begin{cases}
\displaystyle\frac{du}{dt}(t)+\partial _{V}\Phi (u(t))-\partial _{V}\psi
(u(t))=f(t)\;\;\mbox{ in }\;V^{\star },\;\;0<t<T, \\
u(0)=u_{0}.%
\end{cases}
\label{acp}
\end{equation}%
Next, we also define the functional $\phi :L^{2}(\Omega )\rightarrow \lbrack
0,\infty ]$ by
\begin{equation*}
\phi (u):=%
\begin{cases}
\displaystyle\frac{1}{r}\int_{\Omega }|u|^{r}dx\;\; & \mbox{ if }\;u\in
L^{r}(\Omega ) \\
+\infty & \mbox{ otherwise}.%
\end{cases}%
\end{equation*}%
We note that the energy functional $\Phi \left( u\right) -\psi (u)$ is not
bounded from below on $W_{0}^{s,p}(\overline{\Omega})\cap L^{q}(\Omega )$
but the sum $\Phi \left( u\right) -\psi (u)+I_{\mathcal{X}},$ where $I_{%
\mathcal{X}}$ denotes the characteristic function over some ball $\mathcal{X}
$ in $L^{r}\left( \Omega \right) $ turns out to be coercive provided that $r$
satisfies (\ref{cond-7}). In this respect, we can establish the following
crucial result.

\begin{lemma}
\label{lem-33} Let $0<s<1$ and let $p,q,r\in \lbrack 2,\infty )$ satisfy $%
p<q $ and \eqref{cond-7}. Then there exist a constant $\varepsilon \in (0,1]$
and an increasing differentiable function $\mathcal{F}:[0,\infty
)\rightarrow \lbrack 0,\infty )$ such that for every $u\in D(\Phi )\cap
D(\psi )=W_{0}^{s,p}(\overline{\Omega})\cap L^{r}(\Omega )=V$,
\begin{equation}
\psi (u)\leq \mathcal{F}(\phi (u))[\Phi (u)+1]^{1-\varepsilon }.
\label{eq-31-2}
\end{equation}
\end{lemma}

\begin{proof}
Let $u\in W_{0}^{s,p}(\bOm)\cap L^{r}(\Omega )$. If $q\leq r$, since $%
\Omega $ is bounded then by the classical H\"{o}lder inequality, we have
that there exists a constant $C>0$ such that
\begin{equation*}
\frac{1}{q}\int_{\Omega }|u|^{q}dx=\psi (u)\leq C\left( \frac{1}{r}%
\int_{\Omega }|u|^{r}dx\right) ^{\frac{q}{r}}=C[\phi (u)]^{\frac{q}{r}}.
\end{equation*}%
Hence, we have that \eqref{eq-31-2} holds with $\varepsilon =1$ and $%
\mathcal{F}(t)=Ct^{q/r}$.

If $r<q$, then $q<p^{\star }$ (see \eqref{p-star}). Hence, using the
Gagliardo-Nirenberg's inequality \eqref{Gag-Nir-2}, one can find a constant $%
C>0$ such that
\begin{equation}
\Vert u\Vert _{L^{q}(\Omega )}\leq C\left( \frac{C_{N,p,s}}{2p}\int_{\mathbb{%
R}^{N}}\int_{\mathbb{R}^{N}}\frac{|u(x)-u(y)|^{p}}{|x-y|^{N+sp}}dxdy\right)
^{\frac{\alpha }{p}}\Vert u\Vert _{L^{r}(\Omega )}^{1-\alpha },
\label{eq-est}
\end{equation}%
where $\alpha \in (0,1)$ is given by%
\begin{equation}
\frac{1}{q}=\frac{\alpha }{p^{\star }}+\frac{1-\alpha }{r}=\frac{N-sp}{Np}%
\alpha +\frac{1-\alpha }{r}.  \label{eq15}
\end{equation}%
We notice that \eqref{eq15} and \eqref{cond-7} imply that
\begin{equation}
0<\alpha q=(q-r)\left( 1-\frac{N-sp}{Np}r\right) <\frac{\left( \frac{N+sr}{N}%
\right) p-r}{1-\left( \frac{N-sp}{Np}\right) r}=p.  \label{eq16}
\end{equation}%
It follows from \eqref{eq-est} that
\begin{align*}
\psi (u)&\leq C\left( \frac{C_{N,p,s}}{2p}\int_{\mathbb{R}^{N}}\int_{\mathbb{R%
}^{N}}\frac{|u(x)-u(y)|^{p}}{|x-y|^{N+sp}}dxdy\right) ^{\frac{\alpha q}{p}%
}\Vert u\Vert _{L^{r}(\Omega )}^{(1-\alpha )q}\\
&=C[\Phi (u)]^{\frac{\alpha q}{p%
}}[\phi (u)]^{\frac{(1-\alpha )q}{r}}.
\end{align*}%
Note that $0<\alpha q/p<1$ by \eqref{eq16}. Thus we have shown %
\eqref{eq-31-2} with the constant $\varepsilon =1-\alpha q/p$ and $\mathcal{F}(t)=Ct^{%
\frac{(1-\alpha )q}{r}}$. The proof of lemma is finished.
\end{proof}

Next, let $T>0$ be fixed, $u_{0}\in D(\Phi )=V=W_{0}^{s,p}(\overline{\Omega}
)\cap L^{r}(\Omega )$ and $f\in C^{1}([0,T];V)$. We shall introduce an
auxiliary problem associated with the abstract Cauchy problem \eqref{acp}.
To do this, we let $\sigma :=\phi (u_{0})+1$ and set%
\begin{equation*}
V_{\sigma }=\{v\in V:\;\phi (v)\leq \sigma \Longleftrightarrow \left\Vert
v\right\Vert _{L^{r}\left( \Omega \right) }^r\leq r\sigma \}.
\end{equation*}%
We define the proper, convex, lower semi-continuous functional $\Phi
^{\sigma }:V\rightarrow \lbrack 0,\infty ]$ by
\begin{equation*}
\Phi ^{\sigma }(u):=%
\begin{cases}
\Phi (u)\;\; & \mbox{ if }u\in V_{\sigma }, \\
+\infty \;\; & \mbox{ otherwise}.%
\end{cases}%
\end{equation*}%
Clearly, $D(\Phi ^{\sigma })=V_{\sigma }\subset V=D(\Phi )$ and $D(\partial
_{V}\Phi ^{\sigma })=V_{\sigma }\subset V=D(\partial _{V}\Phi )$. It follows
from \cite[Theorem 2.2]{BCP} that for all $u\in D(\partial _{V}\Phi ^{\sigma
})$,
\begin{equation}  \label{Indicator}
\partial _{V}\Phi ^{\sigma }(u)=\partial _{V}\Phi (u)+\partial _{V}\chi
_{V_{\sigma }}(u)
\end{equation}%
where $\chi _{V_{\sigma }}$ denotes the indicator function of the convex set
$V_{\sigma }$ defined by
\begin{equation*}
\chi_{V_\sigma}(u)=%
\begin{cases}
0\;\;\; & u\in V_\sigma \\
\infty & u\not\in V_\sigma.%
\end{cases}%
\end{equation*}
We notice that by \cite[Example 2.8.2]{Br73}, the subdifferential $\partial
_{V}\chi _{V_{\sigma }}$ of the functional $\chi _{V_{\sigma }}$ is given by
\begin{equation}  \label{sub-ind}
\partial _{V}\chi _{V_{\sigma }}(u)=%
\begin{cases}
\emptyset\;\;\; & \mbox{ if }\;u\not\in V_\sigma , \\
\{0\}\;\; & \mbox{ if }\; u\in\mbox{Int}(V_\sigma), \\
\mbox{the normal exterior cone to }\; V_\sigma \;\;\; & \mbox{ if }\; u\in%
\mbox{boundary}(V_\sigma).%
\end{cases}%
\end{equation}

Corresponding to problem (\ref{acp}) we consider the following modified
problem%
\begin{equation}
\begin{cases}
\displaystyle\frac{du}{dt}(t)+\partial _{V}\Phi ^{\sigma }(u(t))-\partial
_{V}\psi (u(t))\ni f(t)\;\;\mbox{ in }\;V^{\star },\;\;0<t<T, \\
u(0)=u_{0}.%
\end{cases}
\label{cp-sigma}
\end{equation}%
We observe that a solution of problem (\ref{cp-sigma}) on $\left( 0,T\right)
$ is also a solution of (\ref{acp}) on $\left( 0,T\right) $ provided that
one has in addition $\phi (u\left( t\right) )< \sigma .$ Indeed, in that
case $\partial _{V}\chi _{V_{\sigma }}(u\left( t\right) )=\left\{ 0\right\}$
by \eqref{sub-ind}, and by \eqref{Indicator}, this implies $\partial
_{V}\Phi ^{\sigma }(u\left( t\right) )=\partial _{V}\Phi (u\left( t\right) )$
a.e. $t\in \left( 0,T\right) $. Thus, in order to establish the existence of
a solution to problem (\ref{acp}) it suffices to construct a sufficiently
regular solution to the Cauchy problem \eqref{cp-sigma} and to derive
additional a priori estimates on this solution. To this end, we first define
the extensions $\overline{\Phi }^{\sigma },\overline{\psi }$ of $\Phi
^{\sigma }$ and $\psi $, respectively, to the Hilbert space $H:=L^{2}(\Omega
)$ by%
\begin{equation*}
\overline{\Phi }^{\sigma }(u):=%
\begin{cases}
\Phi ^{\sigma }(u)\;\; & \mbox{ if }u\in V, \\
+\infty \;\; & \mbox{ otherwise}.%
\end{cases}%
\end{equation*}%
and
\begin{equation*}
\overline{\psi }(u):=%
\begin{cases}
\displaystyle\frac{1}{q}\int_{\Omega }|u|^{q}dx\;\; & \mbox{ if }u\in
L^{q}(\Omega ), \\
+\infty \;\; & \mbox{ otherwise}.%
\end{cases}%
\end{equation*}%
Then, $\overline{\Phi }^{\sigma }$ and $\overline{\psi }$ are proper, convex
and lower semi-continuous on $H=L^{2}(\Omega )$. Let $\partial _{H}\overline{%
\Phi }^{\sigma }$ and $\partial _{H}\overline{\psi }$ denote the
subdifferentials of $\overline{\Phi }^{\sigma }$ and $\overline{\psi }$,
respectively. Then, it readily follows%
\begin{equation}
\begin{cases}
D(\overline{\Phi }^{\sigma })=D({\Phi }^{\sigma }),\;\;D(\partial _{H}%
\overline{\Phi }^{\sigma })\subset D(\partial _{V}{\Phi }^{\sigma }), \\
\partial _{H}\overline{\Phi }^{\sigma }(u)\subset \partial _{V}{\Phi }%
^{\sigma }(u),\;\;\text{for all}\;u\in D(\partial _{H}\overline{\Phi }%
^{\sigma }),%
\end{cases}
\label{eq-12}
\end{equation}%
and
\begin{equation}
\begin{cases}
\overline{\psi }(u)=\psi (u)\;\forall \;u\in V,\;\;D(\partial _{H}\overline{%
\psi })\cap V\subset D(\partial _{V}\psi ), \\
\partial _{H}\overline{\psi }(u)\subset \partial _{V}\psi (u)\;\;\text{for
all}\;u\in D(\partial _{H}\overline{\psi })\cap V.%
\end{cases}
\label{eq-13}
\end{equation}%
Now consider $\overline{\psi }_{\lambda }$ as the Moreau-Yosida
approximation (see \eqref{MY-Apr}) of $\overline{\psi }$, for $\lambda >0$.
Associated with problem (\ref{cp-sigma}), we introduce the following
regularized problem in $H=L^{2}(\Omega )$,

\begin{equation}
\begin{cases}
\displaystyle\frac{du_{\lambda }}{dt}(t)+\partial _{H}\overline{\Phi }%
^{\sigma }(u_{\lambda }(t))-\partial _{H}\overline{\psi }_{\lambda
}(u_{\lambda }(t))\ni f(t)\;\;\mbox{ in }\;H=L^{2}(\Omega ),\;\;0<t<T, \\
u_{\lambda }(0)=u_{0}.%
\end{cases}
\label{cp-sigma-lambda}
\end{equation}

Regarding the functionals defined above, we mention the following facts.

\begin{remark}
\label{rem-43}

\begin{enumerate}
\item It follows from Lemma \ref{lem-33} that for every $u\in D(\Phi
^{\sigma })=V_{\sigma }$,
\begin{equation}
\psi (u)\leq \mathcal{F}[\phi (u)]\left[ \Phi (u)+1\right] ^{1-\varepsilon
}\leq \frac{1}{2}\Phi (u)+\mathcal{F}(\sigma ).  \label{eq17}
\end{equation}

\item There exists a constant $C>0$ such that for every $u\in D(\Phi
^{\sigma })=V_{\sigma }$,
\begin{align}
\Vert u\Vert _{V}^{p}=& \Vert u\Vert _{W_{0}^{s,p}(\overline{\Omega })\cap
L^{r}(\Omega )}^{p}=\left( \Vert u\Vert _{L^{r}(\Omega )}^{2}+\Vert |u\Vert
|_{W_{0}^{s,p}(\overline{\Omega })}^{2}\right) ^{\frac{p}{2}}  \label{eq18}
\\
\leq & C\left( \Vert u\Vert _{L^{r}(\Omega )}^{p}+\Vert |u\Vert
|_{W_{0}^{s,p}(\overline{\Omega })}^{p}\right) \leq C\left( \Phi ^{\sigma
}(u)+\sigma ^{\frac{p}{r}}\right) .  \notag
\end{align}

\item The subdifferential $\partial _{V}\psi :\;V\rightarrow V^{\star }$ is
a compact operator. Indeed, let $C\geq 0$ and let $u_{n}$ be a sequence in $%
V $ such that $\|u_n\|_{V}\le C$. Then, after a subsequence if necessary, $%
u_n$ converges weakly to some $u$ in the reflexive Banach space $V$. Since
the embedding $V\hookrightarrow L^{q}(\Omega )$ is compact, passing to a
subsequence if necessary, we may assume that
\begin{equation*}
u_{n}\rightarrow u\;\text{strongly in}\;L^{q}(\Omega ).
\end{equation*}%
Since $\partial _{V}\psi (u_{n})=|u_{n}|^{q-2}u_{n}$ we have that%
\begin{equation*}
\partial _{V}\psi (u_{n})\rightarrow \partial _{V}\psi (u)\;\text{strongly in%
}\; L^{q^{\prime }}(\Omega ).
\end{equation*}
Since $L^{q^{\prime}}(\Omega)\hookrightarrow V^\star$, it follows that
\begin{equation*}
\partial _{V}\psi (u_{n})\rightarrow \partial _{V}\psi (u)\;\text{strongly in%
}\; V^\star.
\end{equation*}
Hence, $\partial _{V}\psi :\;V\rightarrow V^{\star }$ is a compact operator.

\item Finally, let $J_{\lambda }^{\overline{\psi }}$ ($\lambda >0$) be the
resolvent operator of $\partial _{H}\overline{\psi }$. By Proposition \ref%
{prop-24}, parts (e) and (f), we readily have%
\begin{equation}
\phi (J_{\lambda }^{\overline{\psi }}u)\leq \phi (u)\leq \sigma ,\;\;\;\Phi
(J_{\lambda }^{\overline{\psi }}u)\leq \Phi (u),\;\text{for all}\;u\in
D(\Phi ^{\sigma }).  \label{eq19}
\end{equation}%
Moreover,
\begin{equation}
\Phi ^{\sigma }(J_{\lambda }^{\overline{\psi }}u)\leq \Phi ^{\sigma }(u),\;%
\text{for all}\;u\in D(\Phi ^{\sigma }).  \label{eq20}
\end{equation}
\end{enumerate}
\end{remark}

We conclude this subsection with the following lemma.

\begin{lemma}
\label{lem-10} Recall $0<s<1$ and let $p,q,r\in \lbrack 2,\infty )$ satisfy $%
p<q$ and \eqref{cond-7}. Let $u\in L^{r}(\Omega )$ be such that $v:=|u|^{%
\frac{r-2}{p}}u\in W_0^{s,p}(\overline{\Omega} )$. Let $\overline{\psi }%
_{\lambda }$ and $\phi _{\lambda }$ ($\lambda >0$) be the Morreau-Yosida
approximation of $\overline{\psi }$ and $\phi $, respectively. Then there
exist a constant $\varepsilon \in (0,1]$ and an increasing differentiable
function $\mathcal{F}:\;[0,\infty )\rightarrow \lbrack 0,\infty )$ such that
\begin{align}  \label{E0}
&\int_{\Omega }\partial _{H}\overline{\psi }_{\lambda }(u)\partial _{H}\phi
_{\lambda }(u)dx \\
\leq& \mathcal{F}(\phi (u))\left[ 1+\frac{C_{N,p,s}}{2}\int_{\mathbb{R}^N
}\int_{\mathbb{R}^N }\frac{|v(x)-v(y)|^{p}}{|x-y|^{N+sp}}dxdy\right]
^{1-\varepsilon }.  \notag
\end{align}
\end{lemma}

\begin{proof}
Let $\lambda >0$ and let $u\in L^{r}(\Omega )$ be such that $v:=|u|^{\frac{%
r-2}{p}}u\in W_0^{s,p}(\bOm )$. Since $|J_{\lambda }^{\psi }u(x)|\leq
|u(x)|$ and $|J_{\lambda }^{\phi }u(x)|\leq |u(x)|$ for a.e. $x\in \Omega $,
we have that
\begin{align*}
& \int_{\Omega }\partial _{H}\overline{\psi }_{\lambda }(u)\partial _{H}\phi
_{\lambda }(u)dx \\
=& \int_{\Omega }|J_{\lambda }^{\psi }u(x)|^{q-2}J_{\lambda }^{\psi
}u(x)|J_{\lambda }^{\phi }u(x)|^{r-2}J_{\lambda }^{\phi }u(x)dx\leq
\int_{\Omega }|u(x)|^{q+r-2}dx.
\end{align*}%
In light of \eqref{cond-7} we easily see that%
\begin{equation}
q+r-2<\left( \frac{N+rs}{N}\right) p+r-2=\left( 1+\frac{sp}{N}\right) r+p-2.
\label{E1}
\end{equation}%
If $sp<N$, then%
\begin{align}
\left( \frac{r+p-2}{p}\right) p^{\star }& =\left( \frac{N}{N-sp}\right)
r+\left( \frac{N}{N-sp}\right) (p-2)  \label{E2} \\
& =\left( 1+\frac{sp}{N-sp}\right) r+\left( \frac{N}{N-sp}\right) (p-2).
\notag
\end{align}%
It follows from \eqref{E1} and \eqref{E2} that
\begin{equation*}
\rho :=p\left( \frac{q+r-2}{r+p-2}\right) <p^{\star }.
\end{equation*}

Next, let $v:=|u|^{\frac{r-2}{p}}u$. Then
\begin{equation*}
|u|^{q+r-2}=|v|^{\rho }\;\;\mbox{ and }\;|u|^{r}=|v|^{\frac{pr}{r+p-2}}.
\end{equation*}%
Assume that $v\in W_0^{s,p}(\bOm )$. Since $1<\frac{pr}{r+p-2}<\rho
<p^{\star }$ and
\begin{equation*}
\Vert |v|\Vert _{W_0^{s,p}(\bOm )}\leq \left[ \Vert v\Vert _{L^{\frac{rp}{%
r+p-2}}(\Omega )}+\left( \frac{C_{N,p,s}}{2}\int_{\mathbb{R}^N }\int_{%
\mathbb{R}^N }\frac{|v(x)-v(y)|^{p}}{|x-y|^{N+sp}}\;dxdy\right) ^{\frac 1p}%
\right] ,
\end{equation*}%
then using the Gagliardo-Nirenberg inequality \eqref{Gag-Nir-2}, we have
that there exists a constant $C>0$ such that
\begin{align}
&\Vert v\Vert _{L^{\rho }(\Omega )}\leq  C\Vert |v|\Vert _{W_0^{s,p}(\bOm
)}^{\alpha }\Vert v\Vert _{L^{\frac{rp}{r+p-2}}(\Omega )}^{1-\alpha }
\label{E3} \\
\leq & C\left[ \Vert v\Vert _{L^{\frac{rp}{r+p-2}}(\Omega )}+\left( \frac{%
C_{N,p,s}}{2}\int_{\mathbb{R}^N }\int_{\mathbb{R}^N }\frac{|v(x)-v(y)|^{p}}{%
|x-y|^{N+sp}}dxdy\right) ^{\frac 1p}\right] ^{\alpha }\Vert v\Vert _{L^{%
\frac{rp}{r+p-2}}(\Omega )}^{1-\alpha }  \notag \\
\leq & C\left[ \left( \frac{C_{N,p,s}}{2}\int_{\mathbb{R}^N }\int_{\mathbb{R}%
^N }\frac{|v(x)-v(y)|^{p}}{|x-y|^{N+sp}}\;dxdy\right) ^{\frac 1p}\right]
^{\alpha }\Vert v\Vert _{L^{\frac{rp}{r+p-2}}(\Omega )}^{1-\alpha }+C\Vert
v\Vert _{L^{\frac{rp}{r+p-2}}(\Omega )}  \notag
\end{align}%
with $0\leq \alpha \leq 1$ satisfying
\begin{equation*}
\frac{1}{\rho }=\frac{\alpha }{p^{\star }}+\frac{(1-\alpha )(r+p-2)}{rp}.
\end{equation*}%
A simple calculation gives
\begin{equation*}
0<\alpha =\frac{\left( \frac{r+p-2}{pr}\right) -\left( \frac{r+p-2}{p(q+r-2)}%
\right) }{\left( \frac{r+p-2}{pr}\right) -\left( \frac{N-sp}{Np}\right) }<1.
\end{equation*}%
It follows from \eqref{E3} that,%
\begin{align}
& \int_{\Omega }|u|^{q+r-2}dx=\int_{\Omega }|v|^{\rho }dx  \label{E4} \\
\leq & C\left[ \left( \frac{C_{N,p,s}}{2}\int_{\mathbb{R}^N }\int_{\mathbb{R}%
^N }\frac{|v(x)-v(y)|^{p}}{|x-y|^{N+sp}}dxdy\right) ^{\frac 1p}\right]
^{\alpha \rho }\Vert v\Vert _{L^{\frac{rp}{r+p-2}}(\Omega )}^{(1-\alpha
)\rho }+C\Vert v\Vert _{L^{\frac{rp}{r+p-2}}(\Omega )}^{\rho }  \notag \\
\leq & C\left[ \left( \frac{C_{N,p,s}}{2}\int_{\mathbb{R}^N }\int_{\mathbb{R}%
^N }\frac{|v(x)-v(y)|^{p}}{|x-y|^{N+sp}}dxdy\right) ^{\frac 1p}\right]
^{\alpha \rho }\Vert v\Vert _{L^{r}(\Omega )}^{(1-\alpha )(q+r-2)}+C\Vert
v\Vert _{L^{r}(\Omega )}^{q+r-2}.  \notag
\end{align}%
Letting
\begin{equation*}
\mathcal{F}(t):=\sup\left \{t^{\frac{(1-\alpha )(q+r-2)}{r}},t^{\frac{q+r-2}{r}%
}\right\},\;\;t\geq 0
\end{equation*}%
we have that $\mathcal{F}:\;[0,\infty )\rightarrow \lbrack 0,\infty )$ is
increasing and differentiable. Thus, (\ref{E0}) follows from \eqref{E4}
together with the simple estimate%
\begin{equation*}
1-\varepsilon =\frac{\alpha \rho }{p}=\frac{1}{p}\frac{q-2}{\left( \frac{%
r+p-2}{p}\right) -\left( \frac{N-sp}{Np}\right) }<\frac{1}{p}\frac{\left(
\frac{N+rs}{N}\right) p-2}{\left( \frac{r+p-2}{p}\right) -\left( \frac{N-sp}{%
Np}\right) r}=1.
\end{equation*}%
The proof of the lemma is finished.
\end{proof}

\subsection{Solutions to the auxiliary problems}

In this subsection, we investigate the existence and regularity of solutions
to problems \eqref{cp-sigma}, \eqref{cp-sigma-lambda} for regular initial
datum $u_{0}\in D\left( \Phi \right) =V$ and $f\in C^{1}([0,T];V)$. Before
we turn our attention directly to the Cauchy problem (\ref{cp-sigma-lambda}%
), we require the following two crucial lemmas. The first result is
essential and is of independent interest. The second one establishes a kind
of coercitivity estimate. Their proofs are postponed until the Appendix.

\begin{lemma}
\label{comp}Let $p\in \left( 1,\infty \right) $, $r\in \lbrack 2,\infty )$
and let $\mathcal{E}$ be the energy given by%
\begin{equation*}
\mathcal{E}\left( u,v\right) =\int_{\mathbb{R}^{N}}\int_{\mathbb{R}%
^{N}}|u(x)-u(y)|^{p-2}(u(x)-u(y))(v(x)-v(y))K\left( x,y\right) dxdy,
\end{equation*}%
for some positive kernel $K:\mathbb{R}^{N}\times \mathbb{R}^{N}\rightarrow
\mathbb{R}_{+}$. Then%
\begin{equation}
C_{r,p}\mathcal{E}(|u|^{\frac{r-2}{p}}u,|u|^{\frac{r-2}{p}}u)\leq \mathcal{E}%
(u,\left\vert u\right\vert ^{r-2}u),  \label{comp-energy}
\end{equation}%
for all functions $u$ for which the terms in \eqref{comp-energy} make sense,
and where%
\begin{equation*}
C_{r,p}:=\left( r-1\right) \left( \frac{p}{p+r-2}\right) ^{p}.
\end{equation*}
\end{lemma}

\begin{lemma}
\label{lem-9} Let $0<s<1$, $p,q,r\in \lbrack 2,\infty )$ satisfy $p<q$ and %
\eqref{cond-7}. Let $u\in D(\partial _{V}\Phi )$ and let $J_{\mu }^{\phi
}:=(I+\mu \partial _{H}\phi )^{-1}$, $\mu >0$. Then
\begin{equation*}
J_{\mu }^{\phi }u\in D(\partial _{V}\Phi ),\;\;\partial _{H}\phi _{\mu
}(u)\in V,\;\;v_{\mu }:=|J_{\mu }^{\phi }u|^{\frac{r-2}{p}}J_{\mu }^{\phi
}u\in W_{0}^{s,p}(\overline{\Omega}).
\end{equation*}%
In particular, if $u\in D(\partial _{V}\Phi ^{\sigma })$, then there exists
a positive constant $\beta $ independent of $\mu $ such that for all $g\in
\partial _{V}\Phi ^{\sigma }(u)$,
\begin{equation}
\frac{\beta C_{N,p,s}}{2}\int_{\mathbb{R}^{N}}\int_{\mathbb{R}^{N}}\frac{%
|v_{\mu }(x)-v_{\mu }(y)|^{p}}{|x-y|^{N+sp}}dxdy\leq \langle g,\partial
_{H}\phi _{\mu }(u)\rangle _{V^{\star },V}.  \label{EE0}
\end{equation}
\end{lemma}

We have the following result of existence of solutions to the abstract
Cauchy problem \eqref{cp-sigma-lambda}.

\begin{proposition}
\label{pro-ap-pr} Let $0<s<1$, $p,q,r\in \lbrack 2,\infty )$ satisfy $p<q$
and \eqref{cond-7}. Let $T>0$ be fixed, $u_{0}\in D(\Phi )$, $\lambda >0$
and $f\in C^{1}([0,T];V)$. Then there exists a unique function $u_{\lambda
}\in C_{w}([0,T];V)\cap W^{1,2}((0,T);L^{2}(\Omega ))$ which is a strong
solution of \eqref{cp-sigma-lambda} on $(0,T)$. Moreover,
\begin{equation}
\sup_{t\in \lbrack 0,T]}\phi (u_{\lambda }(t))\leq \sigma \;\;\mbox{ and }%
\;v_{\lambda }:=|u_{\lambda }|^{\frac{r-2}{p}}u_{\lambda }\in
L^{p}((0,T);W_{0}^{s,p}(\overline{\Omega})).  \label{esse}
\end{equation}%
In addition, the function $t\mapsto \overline{\Phi }^{\sigma }(u_{\lambda
}(t))$ is absolutely continuous on $[0,T]$.
\end{proposition}

\begin{proof}
First, we notice that by Proposition \ref{more}, $\partial _{H}\overline{%
\psi }_{\lambda }$ coincides with the Yosida approximation $(\partial _{H}%
\overline{\psi })_{\lambda }$ of the maximal monotone operator $\partial _{H}%
\overline{\psi }$ (see \eqref{Y-Apr}). Hence, by Proposition \ref{prop-24} $%
\partial _{H}\overline{\psi }_{\lambda }$ is Lipschitz continuous in $%
L^{r}(\Omega )$ as well as in $L^{2}(\Omega )$. Since $\overline{\Phi }%
^{\sigma }$ is proper, convex and lower semi-continuous on $H=L^{2}(\Omega )$
and the mapping $t\mapsto f(t)$ belongs to $L^{2}((0,T);L^{2}(\Omega ))$, we
can exploit \cite[Proposition 3.12 and Theorem 3.6]{Br73} to infer that for
every $u_{0}\in \overline{D(\overline{\Phi }^{\sigma })}=L^{2}\left( \Omega
\right) $, the Cauchy problem \eqref{cp-sigma-lambda} has a unique strong
solution $u_{\lambda }$. Moreover, it holds%
\begin{equation*}
\sqrt{t}\frac{du_{\lambda }}{dt}(t)\in L^{2}((0,T);L^{2}(\Omega )).
\end{equation*}%
In particular, if $u_{0}\in D(\overline{\Phi }^{\sigma })$ we have%
\begin{equation*}
u_{\lambda }\in C_{w}([0,T];V)\cap W^{1,2}((0,T);L^{2}(\Omega ))\;\mbox{ and }%
\;u_{\lambda }(t)\in V,\;\;\phi (u_{\lambda }(t))\leq \sigma ,
\end{equation*}%
for all $t\in \lbrack 0,T]$. Hence, the function $t\mapsto \Phi ^{\sigma
}(u_{\lambda }(t))$ is absolutely continuous on $[0,T]$ and the first
statement of (\ref{esse}) also follows. It remains to show the second part
of (\ref{esse}). To this end, multiplying \eqref{cp-sigma-lambda} by $%
\partial _{H}\phi _{\mu }(u_{\lambda }(t))$, $\mu >0$, then employing the
chain rule formula \eqref{eq-chain} (see Proposition \ref{prop-chain}), we
obtain%
\begin{align}
\frac{d}{dt}\phi _{\mu }(u_{\lambda }(t))& +\int_{\Omega }\partial _{H}\phi
_{\mu }(u_{\lambda }(t))\left[ f(t)-\frac{du_{\lambda }}{dt}(t)+\partial _{H}%
\overline{\psi }_{\lambda }(u_{\lambda }(t))\right] dx  \label{eq-326} \\
& =\int_{\Omega }\partial _{H}\overline{\psi }_{\lambda }(u_{\lambda
}(t))\partial _{H}\phi _{\mu }(u_{\lambda }(t))dx+\int_{\Omega }f(t)\partial
_{H}\phi _{\mu }(u_{\lambda }(t))dx.  \notag
\end{align}%
Let $v_{\lambda ,\mu }(t):=|J_{\mu }^{\phi }u_{\lambda }(t)|^{\frac{r-2}{p}%
}J_{\mu }^{\phi }u_{\lambda }(t)$. Note that $v_{\lambda ,\mu }(t)\in
W_{0}^{s,p}(\bOm )\cap L^{r}(\Omega )=V$. Recall that by Lemma \ref{lem-9}
and virtue of \eqref{eq-12} it holds%
\begin{align}
&\beta \frac{C_{N,p,s}}{2}\int_{\mathbb{R}^{N}}\int_{\mathbb{R}^{N}}\frac{%
|v_{\lambda ,\mu }(x,t)-v_{\lambda ,\mu }(y,t)|^{p}}{|x-y|^{N+sp}}dxdy\notag\\
\leq&
\langle \partial _{H}\overline{\Phi }^{\sigma }(u_{\lambda }(t)),\partial
_{H}\phi _{\mu }(u_{\lambda }(t))\rangle _{V^{\star },V}.  \label{326bis}
\end{align}%
By Proposition \ref{prop-24}, $\partial _{H}\overline{\psi }_{\lambda }$ is
Lipschitz continuous from $L^{r}(\Omega )$ to $L^{r}(\Omega )$. Hence, H\"{o}%
lder's inequality together with Proposition \ref{more} yields the estimate%
\begin{align}\label{326tris}
\int_{\Omega }\partial _{H}\overline{\psi }_{\lambda }(u_{\lambda
}(t))\partial _{H}\phi _{\mu }(u_{\lambda }(t))\;dx& \leq \Vert \partial _{H}%
\overline{\psi }_{\lambda }(u_{\lambda }(t))\Vert _{L^{r}(\Omega )}\Vert
\partial _{H}\phi _{\mu }(u_{\lambda }(t))\Vert _{L^{r^{\prime }}(\Omega )}
\notag\\
& \leq C_{\lambda }\phi (u_{\lambda }(t))\leq C_{\lambda }\sigma ,  `
\end{align}%
for some constant $C_{\lambda }>0$ depending only on $\lambda >0$ but not on
$\mu >0$. Moreover, exploiting H\"{o}lder's inequality once again, one can
find a constant $C>0$ such that
\begin{equation*}
\int_{\Omega }f(t)\partial _{H}\phi _{\mu }(u_{\lambda }(t))dx\leq C\sigma ^{%
\frac{1}{r^{\prime }}}\Vert f(t)\Vert _{L^{r}(\Omega )}.
\end{equation*}%
Combining \eqref{326bis} together with (\ref{326tris}), then integrating %
\eqref{eq-326} over $(0,t)$, and using Proposition \ref{more} once more, we
deduce%
\begin{align}
& \phi _{\mu }(u_{\lambda }(t))+\beta \int_{0}^{t}\frac{C_{N,p,s}}{2}\int_{%
\mathbb{R}^{N}}\int_{\mathbb{R}^{N}}\frac{|v_{\lambda ,\mu }(x,\tau
)-v_{\lambda ,\mu }(y,\tau )|^{p}}{|x-y|^{N+sp}}dxdyd\tau  \label{eq-327} \\
\leq & \phi (u_{0})+C_{\lambda }\sigma T+\beta \sigma ^{\frac{1}{r^{\prime }}%
}\int_{0}^{T}\Vert f(\tau )\Vert _{L^{r}(\Omega )}d\tau ,  \notag
\end{align}%
for all $t\in \lbrack 0,T]$. Now passing to the limit as $\mu \rightarrow
0^{+}$ in the foregoing uniform estimate, by virtue of Proposition \ref{more}%
, we obtain%
\begin{equation*}
J_{\mu }^{\phi }\left( u_{\lambda }\right) \rightarrow u_{\lambda }\;\;%
\mbox{ strongly in }\;C([0,T];L^{2}(\Omega )).
\end{equation*}%
Finally, since $2(p+r-2)/p\leq r$, it follows from \eqref{eq-327} that as $%
\mu \rightarrow 0^{+}$,%
\begin{equation*}
\begin{cases}
J_{\mu }^{\phi }\left( u_{\lambda }\right) \rightarrow u_{\lambda }\; &
\text{weakly-star in }L^{\infty }((0,T);L^{r}(\Omega )), \\
v_{\lambda ,\mu }\rightarrow v_{\lambda }\;\; & \text{weakly-star in}%
\;L^{\infty }((0,T);L^{2}(\Omega )), \\
v_{\lambda ,\mu }\rightarrow v_{\lambda } & \text{weakly in}%
\;L^{p}((0,T);W_{0}^{s,p}(\bOm)),%
\end{cases}%
\end{equation*}%
where $v_{\lambda }=|u_{\lambda }|^{\frac{r-2}{p}}u_{\lambda }.$ The proof
is finished.
\end{proof}

Having obtained a solution to the regularized problem (\ref{cp-sigma-lambda}%
), we can now pass to the limit as $\lambda \rightarrow 0^{+}$ to deduce a
solution to problem (\ref{cp-sigma}). We have the following.

\begin{proposition}
Let $0<s<1$, $p,q,r\in \lbrack 2,\infty )$ satisfy $p<q$ and \eqref{cond-7}.
Let $T>0$ be fixed, $u_{0}\in D(\Phi )$ and $f\in C^{1}([0,T];V)$. Then
there exists a unique function $u\in C_{w}([0,T];V)\cap
W^{1,2}((0,T);L^{2}(\Omega ))$ which is a strong solution of problem %
\eqref{cp-sigma} on $(0,T)$.
\end{proposition}

\begin{proof}
Let $u_{0}\in D(\Phi )$ and $f\in C^{1}([0,T];V)$. Let $\lambda >0$ and let $%
u_{\lambda }$ be the unique strong solution of \eqref{cp-sigma-lambda} which
exists by Proposition \ref{pro-ap-pr}. In the subsequent proofs, $C>0$ will
always denote a constant that is independent of $t$, $f$, $\lambda $, which
only depends on the other structural parameters of the problem. Such a
constant may vary even from line to line. We multiply \eqref{cp-sigma-lambda}
by $\frac{du_{\lambda }(t)}{dt}$ and we integrate the resulting identity
over $(0,t)$. Using \eqref{eq17}, \eqref{eq18} and Proposition \ref%
{prop-chain}, we get that there exists a constant $C>0$ such that
\begin{equation}
\int_{0}^{T}\left\Vert \frac{du_{\lambda }}{dt}(t)\right\Vert _{L^{2}(\Omega
)}^{2}dt+\sup_{t\in \lbrack 0,T]}\Phi ^{\sigma }(u_{\lambda }(t))\leq C.
\label{eq21}
\end{equation}%
The estimates \eqref{eq18} and \eqref{eq21} imply that
\begin{equation}
\sup_{t\in \lbrack 0,T]}\Vert u_{\lambda }(t)\Vert _{V}\leq C.
\label{eq22-2}
\end{equation}%
Furthermore, we also have there exists a constant $C>0$ such that
\begin{equation}
\sup_{t\in \lbrack 0,T]}\Vert J_{\lambda }^{\overline{\psi }}u_{\lambda
}(t)\Vert _{V}\leq C  \label{eq22}
\end{equation}%
on account of (\ref{eq20}) and (\ref{eq21}). Let now%
\begin{equation*}
g_{\lambda }:=f(t)-\frac{du_{\lambda }}{dt}(t)+\partial _{H}\overline{\psi }%
_{\lambda }(u_{\lambda }(t))\in \partial _{H}\overline{\Phi }^{\sigma
}(u_{\lambda }(t)).
\end{equation*}%
Then, passing to a subsequence of $\{\lambda \}$ if necessary, we get that
as $\lambda \rightarrow 0^{+}$,
\begin{equation}
\begin{cases}
u_{\lambda }\rightarrow u,\;\;J_{\lambda }^{\overline{\psi }}\left(
u_{\lambda }\right) \rightarrow u\;\; & \mbox{
weakly star in }\;L^{\infty }((0,T);V), \\
u_{\lambda }\rightarrow u,\;\;J_{\lambda }^{\overline{\psi }}\left(
u_{\lambda }\right) \rightarrow u & \mbox{ weakly  in }\;W^{1,2}((0,T);L^{2}(%
\Omega )) \\
\partial _{H}\overline{\psi }_{\lambda }(u_{\lambda }(\cdot ))\rightarrow
\partial _{V}\psi (u(\cdot )) & \mbox{ strongly in }\;C([0,T];L^{q^{\prime
}}(\Omega )), \\
g_{\lambda }\rightarrow g\in \partial _{V}\Phi ^{\sigma }(u(\cdot )) & %
\mbox{ weakly  in }\;L^{2}((0,T);V^{\star }).%
\end{cases}
\label{eq23}
\end{equation}%
The first two foregoing convergence properties follow from \eqref{eq21}, %
\eqref{eq22} and \eqref{eq22-2}. The third convergence property follows from %
\eqref{eq21} in light of Remark \ref{rem-43}, part (c). On the other hand,
the last convergence property follows from the second and third of (\ref%
{eq23}) on the account of the fact that $L^{2}(\Omega )$ and $L^{q^{\prime
}}(\Omega )$ are both continuously embedded into $V^{\star }$. Clearly, %
\eqref{eq23} also yields%
\begin{equation*}
g(t)=f(t)-\frac{du\left( t\right) }{dt}+\partial _{V}\psi (u(t))\text{, a.e.
}t\in (0,T).
\end{equation*}%
Finally, since $u_{\lambda }(t)\rightarrow u_{0}$ strongly in $L^{r}(\Omega
) $ as $t\rightarrow 0^{+}$, we may conclude that the limit function $u$ is
the unique strong solution to the auxiliary problem \eqref{cp-sigma} on $%
(0,T)$. The proof of the proposition is finished.
\end{proof}

\section{Proof of the main results}

\label{pro-main-theo} In this section we prove the main results stated in
Section \ref{sec-main}.

\subsection{Proof of Theorem \protect\ref{main-theo}}

We can now complete the proof of the first main result of the article. This
program will be divided into several steps.\newline

\noindent \textbf{Step 1 (}\textbf{Additional uniform estimates}\textbf{)}.
We give further (uniform) estimates of solutions to the regularized problem %
\eqref{cp-sigma-lambda} that will be needed in the sequel. Recall that $%
p,q,r\in \lbrack 2,\infty )$ satisfy $p<q$ and \eqref{cond-7}. Let $\lambda
>0$ and consider the unique strong solution $u_{\lambda }$ to %
\eqref{cp-sigma-lambda}. Multiplying \eqref{cp-sigma-lambda} by $u_{\lambda
}(t)$, integrating the resulting identity over $(0,t)$ and using \eqref{eq18}%
, we deduce%
\begin{align}
& \frac{1}{2}\Vert u_{\lambda }(t)\Vert _{L^{2}(\Omega )}^{2}+\int_{0}^{t}%
\overline{\Phi }^{\sigma }(u_{\lambda }(\tau ))d\tau  \label{eq33} \\
\leq & \frac{1}{2}\Vert u_{0}\Vert _{L^{2}(\Omega )}^{2}+\int_{0}^{t}\Vert
\partial _{H}\overline{\psi }(u_{\lambda }(\tau ))\Vert _{L^{q^{\prime
}}(\Omega )}\Vert u_{\lambda }(\tau )\Vert _{L^{q}(\Omega )}d\tau  \notag \\
&+\int_{0}^{T}\Vert f(\tau )\Vert _{V^{\star }}\Vert u_{\lambda }(\tau
)\Vert _{V}d\tau  \notag \\
\leq & \frac{1}{2}\Vert u_{0}\Vert _{L^{2}(\Omega )}^{2}+C\int_{0}^{t}\psi
(u_{\lambda }(\tau ))d\tau +C\int_{0}^{T}\Vert f(\tau )\Vert _{V^{\star
}}^{p^{\prime }}d\tau  \notag \\
&+\frac{1}{2}\int_{0}^{t}\overline{\Phi }^{\sigma }(u_{\lambda }(\tau
))d\tau +\frac{T}{2}\sigma ^{\frac{p}{r}}.  \notag
\end{align}%
Lemma \ref{lem-33} together with \eqref{eq17} thus gives%
\begin{align*}
&\sup_{t\in \lbrack 0,T]}\Vert u_{\lambda }(t)\Vert _{L^{2}(\Omega
)}^{2}+\int_{0}^{T}\overline{\Phi }^{\sigma }(u_{\lambda }(\tau ))d\tau \\
&\leq C\left( \Vert u_{0}\Vert _{L^{2}(\Omega )}^{2}+T\mathcal{F}(\sigma
)+\int_{0}^{T}\Vert f(\tau )\Vert _{V^{\star }}^{p^{\prime }}d\tau \right) .
\end{align*}%
Next, multiplying \eqref{cp-sigma-lambda} by $t\frac{du_{\lambda }(t)}{dt}$
and using the fact that
\begin{align*}
\int_{\Omega }f(t)t\frac{du_{\lambda }(t)}{dt}dx=&\frac{d}{dt}\left(
t\int_{\Omega }f(t)u_{\lambda }(t)dx\right) -\int_{\Omega }f(t)u_{\lambda
}(t)dx \\
&-t\int_{\Omega }\frac{df}{dt}(t)u_{\lambda }(t)dx,
\end{align*}%
we obtain%
\begin{align*}
& t\left\Vert \frac{du_{\lambda }(t)}{dt}\right\Vert _{L^{2}(\Omega )}^{2}+%
\frac{d}{dt}\left[ t\overline{\Phi }^{\sigma }(u_{\lambda }(t))\right] -%
\overline{\Phi }^{\sigma }(u_{\lambda }(t)) \\
\leq & \frac{d}{dt}\left[ t\overline{\psi }_{\lambda }(u_{\lambda }(t))%
\right] -\overline{\psi }_{\lambda }(u_{\lambda }(t))+\frac{d}{dt}\left(
t\int_{\Omega }f(t)u_{\lambda }(t)dx\right) \\
& -\int_{\Omega }f(t)u_{\lambda }(t)dx-t\int_{\Omega }\frac{df}{dt}%
(t)u_{\lambda }(t)dx.
\end{align*}%
Integrating the foregoing inequality over $(0,t)$ and using \eqref{eq17} and %
\eqref{eq18} once more, we readily see that%
\begin{align}
& \int_{0}^{t}\tau \left\Vert \frac{du_{\lambda }(\tau )}{d\tau }\right\Vert
_{L^{2}(\Omega )}^{2}d\tau +t\overline{\Phi }^{\sigma }(u_{\lambda
}(t))+\int_{0}^{t}\overline{\psi }_{\lambda }(u_{\lambda }(\tau ))d\tau
\label{eq39b} \\
\leq & t\overline{\psi }_{\lambda }(u_{\lambda }(t))+\int_{0}^{t}\overline{%
\Phi }^{\sigma }(u_{\lambda }(\tau ))d\tau +t\int_{\Omega }f(t)u_{\lambda
}(t)dx  \notag \\
& -\int_{0}^{t}\int_{\Omega }f(\tau )u_{\lambda }(\tau )dxd\tau
-\int_{0}^{t}\tau \int_{\Omega }\frac{df}{d\tau }(t)u_{\lambda }(\tau
)dxd\tau  \notag \\
\leq &\frac{t}{2}\overline{\Phi }^{\sigma }(u_{\lambda }(t))+C\int_{0}^{T}%
\overline{\Phi }^{\sigma }(u_{\lambda }(\tau ))d\tau +C\sup_{\tau \in
\lbrack 0,T]}\tau \Vert f(\tau )\Vert _{V^{\star }}^{p^{\prime }}  \notag \\
& +\frac{t}{4}\overline{\Phi }^{\sigma }(u_{\lambda }(t))+C\int_{0}^{t}\Vert
f(\tau )\Vert _{V^{\star }}^{p^{\prime }}d\tau +\int_{0}^{T}\left\Vert \tau
\frac{df}{d\tau }(\tau )\right\Vert _{V^{\star }}^{p^{\prime }}d\tau +T%
\mathcal{F}(\sigma ).  \notag
\end{align}%
On the other hand, using the fact that
\begin{equation*}
\sup_{t\in \lbrack 0,T]}t\Vert f(t)\Vert _{V^{\star }}^{p^{\prime }}\leq
C\left( \int_{0}^{T}\Vert f(t)\Vert _{V^{\star }}^{p^{\prime
}}\;dt+\int_{0}^{T}\left\Vert t\frac{df}{dt}(t)\right\Vert _{V^{\star
}}^{p^{\prime }}\;dt\right) ,
\end{equation*}%
we further get from (\ref{eq39b}) that
\begin{align}
& \int_{0}^{T}t\left\Vert \frac{du_{\lambda }}{dt}(t)\right\Vert
_{L^{2}(\Omega )}^{2}dt+\sup_{t\in \lbrack 0,T]}t\overline{\Phi }^{\sigma
}(u_{\lambda }(t))  \label{eq40} \\
\leq & C\left( \Vert u_{0}\Vert _{L^{2}(\Omega )}^{2}+T\mathcal{F}(\sigma
)+\int_{0}^{T}\Vert f(t)\Vert _{V^{\star }}^{p^{\prime
}}dt+\int_{0}^{T}\left\Vert t\frac{df}{dt}(t)\right\Vert _{V^{\star
}}^{p^{\prime }}dt\right) .  \notag
\end{align}%
Letting $\lambda \rightarrow 0^{+},$ from \eqref{eq23} and \eqref{eq33} we
infer%
\begin{align}
&\sup_{t\in \lbrack 0,T]}\Vert u(t)\Vert _{L^{2}(\Omega
)}^{2}+\int_{0}^{T}\Phi ^{\sigma }(u(\tau ))d\tau  \notag \\
&\leq C\left( \Vert u_{0}\Vert _{L^{2}(\Omega )}^{2}+T\mathcal{F}(\sigma
)+\int_{0}^{T}\Vert f(t)\Vert _{V^{\star }}^{p^{\prime }}dt\right)
\label{eq41}
\end{align}%
and%
\begin{align}
& \int_{0}^{T}t\left\Vert \frac{du}{dt}(t)\right\Vert _{L^{2}(\Omega
)}^{2}dt+\sup_{t\in \lbrack 0,T]}t\Phi ^{\sigma }(u(t))  \label{eq42} \\
\leq & C\left( \Vert u_{0}\Vert _{L^{2}(\Omega )}^{2}+T\mathcal{F}(\sigma
)+\int_{0}^{T}\Vert f(t)\Vert _{V^{\star }}^{p^{\prime
}}dt+\int_{0}^{T}\left\Vert t\frac{df}{dt}(t)\right\Vert _{V^{\star
}}^{p^{\prime }}dt\right) ,  \notag
\end{align}%
for some constant $C>0$ independent of $t,$ $f$ and $\lambda >0$.\newline

\noindent \textbf{Step 2 (}\textbf{Passage to limit}\textbf{)}. Let $T>0$ be
fixed, $u_{0}\in D(\Phi )$ and $f\in C^{1}([0,T];V)$. Let $\phi _{\mu }$ be
the Moreau-Yosida approximation of $\phi $ for $\mu >0$ and let $u_{\lambda }
$ ($\lambda >0$) be the unique strong solution to \eqref{cp-sigma-lambda}.
Multiplying \eqref{cp-sigma-lambda} by $\partial _{H}\phi _{\mu }(u_{\lambda
}(t))$ and using the chain rule formula (see Proposition \ref{prop-chain}),
we have%
\begin{align*}
& \frac{d}{dt}\phi _{\mu }(u_{\lambda }(t))+\int_{\Omega }g_{\lambda
}(t)\partial _{H}\phi _{\mu }(u_{\lambda }(t))dx \\
& =\int_{\Omega }\partial _{H}\overline{\psi }_{\mu }(u_{\lambda
}(t))\partial _{H}\phi _{\mu }(u_{\lambda }(t))dx+\int_{\Omega }f(t)\partial
_{H}\phi _{\mu }(u_{\lambda }(t))dx.
\end{align*}%
Let $v_{\lambda ,\mu }\left( t\right) :=|J_{\mu }^{\Phi }u_{\lambda }\left(
t\right) |^{\frac{r-2}{p}}J_{\mu }^{\Phi }u_{\lambda }\left( t\right) $ for
a.e. $t\in (0,T)$ and note that $v_{\lambda ,\mu }(t)\in W_{0}^{s,p}(%
\overline{\Omega })$ on account of Lemma \ref{lem-9}. Inserting the
estimates of Lemmas \ref{lem-9} and \ref{lem-10} into the foregoing
identity, we readily have%
\begin{align}
& \frac{d}{dt}\Big[\phi _{\mu }(u_{\lambda }(t))\Big]+\frac{\beta C_{N,p,s}}{%
2}\int_{\mathbb{R}^{N}}\int_{\mathbb{R}^{N}}\frac{|v_{\lambda ,\mu
}(x,t)-v_{\lambda ,\mu }(y,t)|^{p}}{|x-y|^{N+sp}}dxdy  \label{eq30} \\
\leq & \mathcal{F}(\sigma )\left[ \frac{C_{N,p,s}}{2}\int_{\mathbb{R}%
^{N}}\int_{\mathbb{R}^{N}}\frac{|v_{\lambda ,\mu }(x,t)-v_{\lambda ,\mu
}(y,t)|^{p}}{|x-y|^{N+sp}}dxdy\right] ^{1-\varepsilon }  \notag \\
& +\int_{\Omega }f(t)\partial _{H}\phi _{\mu }(u_{\lambda }(t))dx.  \notag
\end{align}%
H\"{o}lder's inequality and \eqref{eq19} allow us to deduce%
\begin{align}
\int_{\Omega }f(t)\partial _{H}\phi _{\mu }(u_{\lambda }(t))\;dx\leq & \Vert
f(t)\Vert _{L^{r}(\Omega )}\Vert \partial _{H}\phi _{\mu }(u_{\lambda
}(t))\Vert _{L^{r^{\prime }}(\Omega )}  \label{eq30-2} \\
\leq & C\Vert f(t)\Vert _{L^{r}(\Omega )}\phi (J_{\lambda }^{\phi
}(u_{\lambda }(t)))^{\frac{1}{r^{\prime }}}  \notag \\
\leq & C\sigma ^{\frac{1}{r^{\prime }}}\Vert f(t)\Vert _{L^{r}(\Omega )}.
\notag
\end{align}%
Integrating \eqref{eq30} over $(0,t)$, using \eqref{eq30-2}, and recalling
that the function $v_{\lambda ,\mu }\in L^{p}((0,T);W_{0}^{s,p}(\overline{%
\Omega }))$, there holds%
\begin{align}
& \phi _{\mu }(u_{\lambda }(t))+\beta \int_{0}^{t}\frac{C_{N,p,s}}{2}\int_{%
\mathbb{R}^{N}}\int_{\mathbb{R}^{N}}\frac{|v_{\lambda ,\mu }(x,\tau
)-v_{\lambda ,\mu }(y,\tau )|^{p}}{|x-y|^{N+sp}}dxdyd\tau   \label{eq31} \\
\leq & \phi _{\mu }(u_{0})+\mathcal{F}(\sigma )\int_{0}^{t}\left[ \frac{%
C_{N,p,s}}{2}\int_{\mathbb{R}^{N}}\int_{\mathbb{R}^{N}}\frac{|v_{\lambda
,\mu }(x,\tau )-v_{\lambda ,\mu }(y,\tau )|^{p}}{|x-y|^{N+sp}}dxdy\right]
^{1-\varepsilon }d\tau   \notag \\
& +C\sigma ^{\frac{1}{r^{\prime }}}\int_{0}^{t}\Vert f(\tau )\Vert
_{L^{r}(\Omega )}d\tau ,  \notag
\end{align}%
for some $C>0$ independent of $\mu >0$. Therefore, since $\phi _{\mu
}(u_{0})\leq \phi (u_{0})$, after passing to a subsequence of $\{\mu \}$ if
necessary, we can infer the existence of a function $w_{\lambda }\in
L^{p}((0,T);W_{0}^{s,p}(\overline{\Omega }))$ such that, as $\mu \rightarrow
0^{+}$,
\begin{equation}
v_{\lambda ,\mu }\rightarrow w_{\lambda }\;\mbox{ weakly in }%
\;L^{p}((0,T);W_{0}^{s,p}(\overline{\Omega })).  \label{eq32}
\end{equation}%
Next, it follows from Proposition \ref{more} that
\begin{equation}
\frac{1}{2\mu }\left\Vert u_{\lambda }(t)-J_{\mu }^{\phi }u_{\lambda
}(t)\right\Vert _{L^{2}(\Omega )}^{2}=\phi _{\mu }(u_{\lambda }(t))-\phi
(J_{\mu }^{\phi }u_{\lambda }(t))\leq \sigma .  \label{E-mu}
\end{equation}%
Estimate \eqref{E-mu} implies that $J_{\mu }^{\phi }u_{\lambda
}(t)\rightarrow u_{\lambda }$ strongly in $C([0,T];L^{2}(\Omega )),$ as $\mu
\rightarrow 0^{+}$. Hence, by \eqref{eq32} we can deduce that $w_{\lambda
}=v_{\lambda }=|u_{\lambda }|^{\frac{r-2}{p}}u_{\lambda }$. Moreover, by
lower semi-continuity it follows that%
\begin{align*}
& \int_{0}^{t}\frac{C_{N,p,s}}{2}\int_{\mathbb{R}^{N}}\int_{\mathbb{R}^{N}}%
\frac{|v_{\lambda }(x,\tau )-v_{\lambda }(y,\tau )|^{p}}{|x-y|^{N+sp}}%
dxdyd\tau  \\
\leq & \liminf_{\mu \rightarrow 0^{+}}\int_{0}^{t}\frac{C_{N,p,s}}{2}\int_{%
\mathbb{R}^{N}}\int_{\mathbb{R}^{N}}\frac{|v_{\lambda ,\mu }(x,\tau
)-v_{\lambda ,\mu }(y,\tau )|^{p}}{|x-y|^{N+sp}}dxdyd\tau .
\end{align*}%
Passing to the limit in \eqref{eq31} with respect to $\mu \rightarrow 0^{+}$%
, and applying Young's inequality, we get that%
\begin{align}
& \phi (u_{\lambda }(t))+\frac{\beta }{2}\int_{0}^{t}\frac{C_{N,p,s}}{2}%
\int_{\mathbb{R}^{N}}\int_{\mathbb{R}^{N}}\frac{|v_{\lambda }(x,\tau
)-v_{\lambda }(y,\tau )|^{p}}{|x-y|^{N+sp}}dxdyd\tau   \label{eq34} \\
\leq & \phi (u_{0})+t\mathcal{F}(\sigma )+t\frac{\beta }{2}+C\sigma ^{\frac{1%
}{r^{\prime }}}\int_{0}^{t}\Vert f(\tau )\Vert _{L^{r}(\Omega )}d\tau ,
\notag
\end{align}%
for all $t\in \lbrack 0,T]$. Next, since the embedding $V\hookrightarrow
L^{q}(\Omega )$ is compact, the application of Ascoli's compactness lemma
together with \eqref{eq23} yields%
\begin{equation*}
u_{\lambda }\rightarrow u\,\mbox{ strongly in }\;C([0,T];L^{q}(\Omega )).
\end{equation*}%
Since $\phi (u_{\lambda }(t))\leq \sigma $ for all $t\in \lbrack 0,T]$ and $%
2\left( p+r-2\right) /p\leq r$, then letting $\lambda \rightarrow 0^{+}$ in %
\eqref{eq34}, we also obtain that%
\begin{equation}
\begin{cases}
u_{\lambda }\rightarrow u\,\; & \mbox{ weakly star in }\;L^{\infty
}((0,T);L^{r}(\Omega )), \\
v_{\lambda }\rightarrow v\, & \mbox{ weakly star in }\;L^{\infty
}((0,T);L^{2}(\Omega )), \\
v_{\lambda }\rightarrow v\text{ } & \mbox{ weakly in }%
\;L^{p}((0,T);W_{0}^{s,p}(\overline{\Omega })),%
\end{cases}
\label{eq35}
\end{equation}%
where $v_{\lambda }:=|u_{\lambda }|^{\frac{r-2}{p}}u_{\lambda }$ and $%
v:=|u|^{\frac{r-2}{p}}u$. We can then conclude from \eqref{eq34} and %
\eqref{eq35} that there exists a constant $C>0$ such that
\begin{align}
& \phi (u(t))+\frac{\beta }{2}\int_{0}^{t}\frac{C_{N,p,s}}{2}\int_{\mathbb{R}%
^{N}}\int_{\mathbb{R}^{N}}\frac{|v(x,\tau )-v(y,\tau )|^{p}}{|x-y|^{N+sp}}%
dxdyd\tau   \label{eq36} \\
\leq & \phi (u_{0})+t\mathcal{F}(\sigma )+t\frac{\beta }{2}+C\sigma ^{\frac{1%
}{r^{\prime }}}\int_{0}^{t}\Vert f(\tau )\Vert _{L^{r}(\Omega )}d\tau .
\notag
\end{align}%
The final estimate \eqref{eq36} implies that%
\begin{equation}
\limsup_{t\rightarrow 0^{+}}\phi (u(t))\leq \phi (u_{0}).  \label{eq37}
\end{equation}%
Since $u\in C([0,T];L^{2}(\Omega ))$ and $\phi $ is lower semi-continuous,
we have that%
\begin{equation}
\liminf_{t\rightarrow 0^{+}}\phi (u(t))\geq \phi (u_{0}).  \label{eq37bis}
\end{equation}%
Since $L^{r}(\Omega )$ is uniformly convex, we obtain from \eqref{eq37} and %
\eqref{eq37bis} that
\begin{equation}
u(t)\rightarrow u_{0}\;\mbox{ strongly in }\;L^{r}(\Omega )\;\mbox{ as }%
\;t\rightarrow 0^{+}.  \label{eq37tris}
\end{equation}

\noindent \textbf{Step 3} (\textbf{Solution to the original problem}) Let $%
T>0$ be fixed, $u_{0}\in D(\Phi )$ and%
\begin{equation}
f\in W^{1,p^{\prime }}(0,T;W^{-s,p^{\prime }}(\Omega )+L^{r^{\prime
}}(\Omega ))\cap L^{1+\gamma }(0,T;L^{r}(\Omega ))=:\mathcal{Y}_{f}
\label{fff}
\end{equation}%
for some $\gamma \geq 0$. Let $\mathcal{F}:[0,\infty )\rightarrow \lbrack
0,\infty )$ be an increasing differentiable function satisfying the
conclusion of Lemmas \ref{lem-33} and \ref{lem-10}.

\begin{itemize}
\item If $\gamma >0$, we take a non-increasing function $T_{\star
}:[0,\infty )\times \lbrack 0,\infty )\rightarrow (0,T]$ independent of $T$,
$u_{0}$ and $f$ such that
\begin{equation*}
T_{\star }(\eta ,\xi )\left[ \mathcal{F}(\eta +1)+\frac{\beta }{2}\right]
+C(\eta +1)^{\frac{1}{r^{\prime }}}T_{\star }(\eta ,\xi )^{\frac{\gamma }{%
\gamma +1}}\xi ^{\frac{1}{1+\gamma }}\leq \frac{1}{2}.
\end{equation*}

\item If $\gamma =0$, we take a non-increasing function $T_{f}:[0,\infty
)\rightarrow (0,T]$ which depends on $f$ but not on $T$ and $u_{0}$ such that%
\begin{equation*}
T_{f}(\eta )\left[ \mathcal{F}(\eta +1)+\frac{\beta }{2}\right] +C(\eta +1)^{%
\frac{1}{r^{\prime }}}\int_{0}^{T_{\star }(\eta )}\Vert f(\tau )\Vert
_{L^{r}(\Omega )}\leq \frac{1}{2}.
\end{equation*}
\end{itemize}

Let now%
\begin{equation*}
T_{0}:=T_{\star }\left( \phi (u_{0}),\int_{0}^{T}\Vert f(\tau )\Vert
_{L^{r}(\Omega )}^{1+\gamma }\;d\tau \right) >0\;\mbox{ if }\;\gamma >0
\end{equation*}%
and
\begin{equation*}
T_{0}:=T_{f}(\phi (u_{0}))\;\mbox{ if }\;\gamma =0.
\end{equation*}%
Since $\sigma =\phi (u_{0})+1$, it follows that%
\begin{equation*}
\sup_{t\in \lbrack 0,T_{0}]}\phi (u(t))<\sigma .
\end{equation*}%
Since $\phi (u(t))<\sigma $ for all $t\in \lbrack 0,T_{0}]$, it follows from %
\eqref{sub-ind} that $\partial _{V}\chi _{V_{\sigma }}(u(t))=\{0\},$ a.e. $%
t\in \lbrack 0,T_{0}];$ thus by \eqref{Indicator}, $\partial _{V}\Phi
^{\sigma }(u(t))=\partial _{V}\Phi (u(t))$ for a.e. $t\in \lbrack 0,T_{0}]$.
We have shown that $u$ is a strong solution of \eqref{acp} on $(0,T_{0})$
and hence, a strong solution of \eqref{par-pro} on $(0,T_{0})$ if the
initial datum $u_{0}\in D\left( \Phi \right) $.\newline

\noindent \textbf{Step 4} (\textbf{Final argument}). In this final step, we
remove the assumption on the initial datum $u_{0}\in D(\Phi )$. To this end,
for fixed time $T>0$, consider $u_{0}\in L^{r}(\Omega )$ and a function $f$
satisfying (\ref{fff}). Let $u_{0,n}\in D(\Phi )$ and $f_{n}\in
C^{1}([0,T];V)$ be sequences such that $u_{0,n}\rightarrow u_{0}$ strongly
in $L^{r}(\Omega )$ and $f_{n}\rightarrow f$ strongly in $\mathcal{Y}_{f}$.
Let $\sigma :=\phi (u_{0})+2$. Then for sufficiently large $n\geq n_{0}$, we
have%
\begin{equation*}
\phi (u_{0,n})\leq \phi (u_{0})+1\;\mbox{ and }\;\int_{0}^{T}\Vert
f_{n}(t)\Vert _{L^{r}(\Omega )}^{1+\gamma }dt\leq \int_{0}^{T}\Vert
f(t)\Vert _{L^{r}(\Omega )}^{1+\gamma }dt+1.
\end{equation*}%
Moreover, there exists a function $h\in L^{1+\gamma }(0,T)$ such that, after
passing to a subsequence if necessary, we also have%
\begin{equation*}
\Vert f_{n}(t)\Vert _{L^{r}(\Omega )}\leq h(t)\;\mbox{ for a.e. }\;t\in
(0,T).
\end{equation*}

We now consider the $n$-approximate problem%
\begin{equation}
\begin{cases}
\displaystyle\frac{du_{n}}{dt}(t)+\partial _{V}\Phi (u_{n}(t))-\partial
_{V}\psi (u_{n}(t))=f_{n}(t),\;\;\;\  & t\in (0,T) \\
u_{n}(0)=u_{0,n}. &
\end{cases}
\label{par-pro-n}
\end{equation}%
Note that \eqref{par-pro-n} possesses a strong solution $u_{n}$ on $%
(0,T_{0}) $ satisfying
\begin{equation}
\sup_{t\in \lbrack 0,T_{0}]}\phi (u_{n}(t))\leq \phi (u_{0,n})+1\leq \sigma
\label{eq43}
\end{equation}%
for some $T_{0}$ independent of $n$. Indeed, employing \textbf{Step 1} once
again, we have the following alternatives.

\begin{itemize}
\item If $\gamma >0$, then it is clear that
\begin{equation*}
T_{\star }\left( \phi (u_{0,n}),\int_{0}^{T}\Vert f_{n}\Vert _{L^{r}(\Omega
)}^{1+\gamma }dt\right) \geq T_{\star }\left( \Phi
(u_{0})+1,\int_{0}^{T}\Vert f\Vert _{L^{r}(\Omega )}^{1+\gamma }dt+1\right) .
\end{equation*}

\item If $\gamma =0$, since $\Vert f_{n}(t)\Vert _{L^{r}(\Omega )}\leq h(t)$%
, then we can choose the function $T_{h}:[0,\infty )\rightarrow (0,T]$ such
that
\begin{equation*}
T_{h}(\eta )\left[ \mathcal{F}(\eta +1)+\frac{\beta }{2}\right] +C(\eta +1)^{%
\frac{1}{r^{\prime }}}\int_{0}^{T_{h}(\eta )}|h(\tau )|d\tau \leq \frac{1}{2}%
\;\mbox{ and }\;T_{f_{n}}(\eta )\geq T_{h}(\eta ),
\end{equation*}%
for any $\eta \in \lbrack 0,\infty ).$
\end{itemize}

Hence, we can take $T_{0}>0$ uniformly with respect to $n$. In the remainder
of the proof, $C>0$ will denote a constant that is independent of $t$, $f$, $%
n$, and initial data, which only depends on the other structural parameters
of the problem. Such a constant may vary even from line to line. It remains
to derive uniform estimates for the solution $u_{n}$ with respect to $n$.
First, by estimates \eqref{eq41} and \eqref{eq42},%
\begin{equation}
\sup_{t\in \lbrack 0,T_{0}]}\Vert u_{n}(t)\Vert _{L^{2}(\Omega
)}+\int_{0}^{T_{0}}\frac{C_{N,p,s}}{2}\int_{\mathbb{R}^{N}}\int_{\mathbb{R}%
^{N}}\frac{|u_{n}(x,t)-u_{n}(y,t)|^{p}}{|x-y|^{N+sp}}dxdydt\leq C
\label{eq44}
\end{equation}%
and
\begin{equation}
\int_{0}^{T_{0}}t\left\Vert \frac{du_{n}}{dt}(t)\right\Vert _{L^{2}(\Omega
)}^{2}dt+\sup_{t\in \lbrack 0,T_{0}]}\frac{tC_{N,p,s}}{2}\int_{\mathbb{R}%
^{N}}\int_{\mathbb{R}^{N}}\frac{|u_{n}(x,t)-u_{n}(y,t)|^{p}}{|x-y|^{N+sp}}%
dxdy\leq C.  \label{eq45}
\end{equation}%
Since $\partial _{V}\Phi (u_{n}(t))=(-\Delta )_{p}^{s}u_{n}(t)$ (see
Proposition \ref{subdif}), then using \eqref{eq44} and the fact that (see
the proof of Proposition \ref{subdif})
\begin{align*}
\langle (-\Delta)_p^su_n,u_n\rangle_{V^\star,V}=\frac{C_{N,p,s}}{2}\int_{%
\mathbb{R} ^{N}}\int_{\mathbb{R}^{N}}\frac{|u_{n}(x,t)-u_{n}(y,t)|^{p}}{%
|x-y|^{N+sp}}dxdy
\end{align*}
we infer
\begin{equation}
\int_{0}^{T_{0}}\Vert \partial _{V}\Phi (u_{n}(t))\Vert _{W^{-s,p^{\prime
}}(\Omega )}^{p^{\prime }}dt\leq C.  \label{eq46}
\end{equation}%
Application of Lemma \ref{lem-33} also yields%
\begin{align}
\int_{0}^{T_{0}}\Vert \partial _{V}\Phi (u_{n}(t))\Vert _{L^{q^{\prime
}}(\Omega )}^{q^{\prime }}dt\leq& C\int_{0}^{T_{0}}\psi (u_{n}(t))dt  \notag
\\
\leq &\mathcal{F}(\sigma )\int_{0}^{T_{0}}\left( \Phi (u_{n}(t))+1\right)
^{1-\varepsilon }dt\leq C.  \label{eq47}
\end{align}%
Therefore, since $L^{q^{\prime }}(\Omega )\hookrightarrow V^{\star }$, it
follows from \eqref{par-pro-n} that
\begin{equation}
\int_{0}^{T_{0}}\left\Vert \frac{du_{n}}{dt}(t)\right\Vert _{V^{\star
}}^{q^{\prime }}dt\leq C.  \label{eq48}
\end{equation}%
Estimate \eqref{eq36} with $u=u_{n}$, $v=v_{n}=|u_{n}|^{\frac{r-2}{p}}u_{n}$%
, $u_{0}=u_{0,n}$ and $f=f_{n}$ gives the uniform estimate%
\begin{equation}
\int_{0}^{T_{0}}\frac{C_{N,p,s}}{2}\int_{\mathbb{R}^{N}}\int_{\mathbb{R}^{N}}%
\frac{|v_{n}(x,t)-v_{n}(y,t)|^{p}}{|x-y|^{N+sp}}dxdydt\leq C.  \label{eq49}
\end{equation}%
Since $2(p+r-2)/r\leq r$, it follows from \eqref{eq43} that the sequence $%
v_{n}$ is bounded in $L^{\infty }((0,T_{0});L^{2}(\Omega ))$. We also notice
the solution $u_{n}$ satisfies the estimates \eqref{eq41} and \eqref{eq42}
with a constant $C>0$ independent of $n$. These uniform estimates allow us
to pass to the limit, after a subsequence if necessary, such that as $%
n\rightarrow \infty $,%
\begin{equation}
\begin{cases}
u_{n}\rightarrow u\; & \mbox{ weakly star in }\;L^{\infty
}((0,T_{0});L^{r}(\Omega )), \\
u_{n}\rightarrow u & \mbox{ weakly in }\;L^{p}((0,T_{0});V), \\
t^{\frac{1}{p}}u_{n}\rightarrow t^{\frac{1}{p}}u & \mbox{ weakly start in }%
\;L^{\infty }((0,T_{0});W_{0}^{s,p}(\overline{\Omega} )), \\
v_{n}\rightarrow v & \mbox{ weakly star in }\;L^{\infty
}((0,T_{0});L^{2}(\Omega )), \\
v_{n}\rightarrow v & \mbox{ weakly  in }\;L^{p}((0,T_{0});W_{0}^{s,p}(%
\overline{\Omega} )), \\
\frac{du_{n}}{dt}\rightarrow \frac{du}{dt} & \mbox{ weakly  in }%
\;L^{q^{\prime }}((0,T_{0});V^{\star }), \\
\sqrt{t}\frac{du_{n}}{dt}\rightarrow \sqrt{t}\frac{du}{dt} &
\mbox{ weakly
in }\;L^{2}((0,T_{0});L^{2}(\Omega )), \\
\partial _{V}\Phi (u_{n}(\cdot ))\rightarrow g\left( \cdot \right) &
\mbox{
weakly  in }\;L^{p^{\prime }}((0,T_{0});W^{-s,p^{\prime }}(\Omega )), \\
\partial _{V}\psi (u_{n}(\cdot ))\rightarrow h\left( \cdot \right) &
\mbox{
weakly  in }\;L^{q^{\prime }}((0,T_{0});L^{q^{\prime }}(\Omega )).%
\end{cases}
\label{eq52}
\end{equation}%
The first two of (\ref{eq52}) follow from \eqref{eq43} and \eqref{eq44}. The
third and seventh convergence properties of (\ref{eq52}) follow from %
\eqref{eq45}. The fourth and fifth convergence properties are derived from
the fact that $v_{n}$ is bounded in $L^{\infty }((0,T_{0});L^{2}(\Omega ))$
and from \eqref{eq48}. The sixth convergence is an immediate consequence of %
\eqref{eq47}, while the convergence $\partial _{V}\Phi (u_{n}(\cdot
))\rightarrow g$ is a consequence of \eqref{eq46}. Finally, the last of (\ref%
{eq52}) follows from the sixth of (\ref{eq52}) and $\partial _{V}\Phi
(u_{n}(\cdot ))\rightarrow g$. Thus, we have shown%
\begin{equation*}
u\in C_{w}([0,T_{0}];L^{r}(\Omega ))\cap C((0,T_{0}];L^{2}(\Omega )).
\end{equation*}
We can now pass to strong convergence properties for the sequence $u_{n}$.
Since the embeddings $V\hookrightarrow L^{q}(\Omega )$ and $L^{r}(\Omega
)\hookrightarrow V^{\star }$ are compact, it follows that
\begin{equation}
u_{n}\rightarrow u\;\mbox{ strongly in }\;L^{p}((0,T_{0});L^{q}(\Omega
))\cap C([0,T_{0}];V^{\star }),  \label{eq58}
\end{equation}%
which together with \eqref{eq52} implies that $v=|u|^{\frac{r-2}{p}}u$.
Moreover, it follows from \eqref{eq47} and \eqref{eq58} that $%
u(t)\rightarrow u_{0}$ strongly in $V^{\star }$ as $t\rightarrow 0^{+}$.

It remains to show that $\partial _{V}\psi (u(t))=h(t)$ and $g(t)=\partial
_{V}\Phi (u(t))$ for a.e. $t\in (0,T_{0})$. Indeed, if $r<q$, as in the
proof of Lemma \ref{lem-33} we have using \eqref{eq16} that
\begin{align}
& \int_{0}^{T}\Vert u_{n}(t)-u(t)\Vert _{L^{q}(\Omega )}^{q}dt  \label{eq59}
\\
\leq& C\left( \int_{0}^{T_{0}}\frac{C_{N,p,s}}{2}\int_{\mathbb{R}^{N}}\int_{%
\mathbb{R}^{N}}\frac{|(u_{n}-u)(x,t)-(u_{n}-u)(y,t)|^{p}}{|x-y|^{N+sp}}%
dxdydt\right) ^{\frac{\alpha q}{p}}  \notag \\
&\cdot\left( \int_{0}^{T_{0}}\Vert u_{n}(t)-u(t)\Vert _{L^{r}(\Omega
)}^{(1-\alpha )q\nu }dt\right) ^{\frac{1}{\nu }},  \notag
\end{align}%
with $\alpha >0$ given by \eqref{eq15} and $\nu =\frac{p}{p-\alpha q}$. It
follows from \eqref{eq43} and \eqref{eq58} that
\begin{equation*}
u_{n}\rightarrow u\;\mbox{ strongly in }L^{(1-\alpha )q\nu
}((0,T_{0});L^{r}(\Omega )).
\end{equation*}%
and, from \eqref{eq44} and \eqref{eq59}, that
\begin{equation}
u_{n}\rightarrow u\;\mbox{ strongly in }L^{q}((0,T_{0});L^{q}(\Omega )).
\label{eq60}
\end{equation}%
We notice that $\partial _{V}\psi (u)=\partial _{L^{q}(\Omega )}\psi
_{L^{q}}(u)$ if $u\in V$, where $\psi _{L^{q}}:\;L^{q}(\Omega )\rightarrow
\lbrack 0,\infty )$ is defined by%
\begin{equation*}
\psi _{L^{q}}(u):=\frac{1}{q}\int_{\Omega }|u|^{q}dx,\qquad \text{ for all }%
u\in L^{q}(\Omega ).
\end{equation*}
Since the subdifferential $\partial _{L^{q}(\Omega )}\psi _{L^{q}}$ is
demi-closed in $L^{q}(\Omega )\times L^{q^{\prime }}(\Omega )$, we can apply
\cite[Proposition 1.1]{Ken} to infer that $h(t)=\partial _{V}\psi (u(t))$,
a.e. $t\in (0,T_{0})$. If $q\leq r$, then \eqref{eq60} follows from %
\eqref{eq43} and \eqref{eq58}. Hence, we have shown the first claim that $%
h(t)=\partial _{V}\psi (u(t))$ for a.e. $t\in (0,T_{0})$. In order to show
that $g(t)=\partial _{V}\Phi (u(t)),$ a.e. $t\in (0,T_{0})$, we use (\ref%
{eq58}) to take a set $I\subset (0,T_{0})$ such that $u_{n}(\tau
)\rightarrow u(\tau )$ strongly on $L^{q}(\Omega )$ for all $\tau \in I$ and
$|(0,T_{0})\setminus I|=0$. Hence, for all $\tau \in I$,
\begin{align*}
\int_{\tau }^{T_{0}}\langle \partial _{V}\Phi
(u_{n}(t)),u_{n}(t)\rangle_{V^\star,V} dt=& \int_{\tau }^{T_{0}}\langle
f_{n}(t),u_{n}(t)\rangle_{V^\star,V} dt \\
&+\int_{\tau }^{T_{0}}\langle \partial _{V}\psi
(u_{n}(t)),u_{n}(t)\rangle_{V^\star,V} dt \\
& -\frac{1}{2}\Vert u_{n}(T_{0})\Vert _{L^{2}(\Omega )}^{2}+\frac{1}{2}\Vert
u_{n}(\tau )\Vert _{L^{2}(\Omega )}^{2}.
\end{align*}%
Since by (\ref{eq52}), $u\in W^{1,2}((\tau ,T_{0});L^{2}(\Omega ))$, then
letting $n\rightarrow \infty $ in the preceding equality, we deduce%
\begin{align*}
\limsup_{n\rightarrow \infty }\int_{\tau }^{T_{0}}\langle \partial _{V}\Phi
(u_{n}(t)),u_{n}(t)\rangle _{V^{\star },V}dt\leq & \int_{\tau
}^{T_{0}}\langle f(t),u(t)\rangle _{V^{\star },V}dt \\
&+\int_{\tau }^{T_{0}}\langle \partial _{V}\psi (u(t)),u(t)\rangle
_{V^{\star },V}dt \\
& -\frac{1}{2}\Vert u(T_{0})\Vert _{L^{2}(\Omega )}^{2}+\frac{1}{2}\Vert
u(\tau )\Vert _{L^{2}(\Omega )}^{2} \\
& =\int_{\tau }^{T_{0}}\langle g(t),u(t)\rangle _{V^{\star },V}dt.
\end{align*}%
It follows from \eqref{eq52} that $g(t)=\partial _{V}\Phi (u(t))$, a.e. $%
t\in (\tau ,T_{0})$. Since $\tau $ was arbitrary and $|(0,T_{0})\setminus
I|=0$, we have that $g(t)=\partial _{V}\Phi (u(t))$, a.e. $t\in (0,T_{0})$.

It remains to show that $u(0)=u_{0}$ in the sense of $L^{r}(\Omega )$.
Estimate \eqref{eq36} with $u=u_{n}$, $v=v_{n}$, $u_{0}=u_{0,n}$ and $%
f=f_{n} $ allows us to pass to the limit as $n\rightarrow \infty $, to get%
\begin{equation*}
\phi (u(t))\leq \phi (u_{0})+t\left[ \mathcal{F}(\sigma )+\frac{\beta }{2}%
\right] +C\sigma ^{\frac{1}{r^{\prime }}}\int_{0}^{t}\Vert f(\tau )\Vert
_{L^{r}(\Omega )}d\tau ,
\end{equation*}%
for all $t\in \lbrack 0,T_{0}]$. Arguing exactly as in (\ref{eq37})-(\ref%
{eq37tris}) we easily find that $u(t)\rightarrow u_{0}$ strongly in $%
L^{r}(\Omega )$ as $t\to 0^+$ and $u\in C([0,T_{0}];L^{2}(\Omega ))$. We
have shown that $u $ is a strong solution to problem \eqref{acp} on $%
(0,T_{0})$ and hence, a strong solution to the initial-boundary value
problem \eqref{par-pro} on $(0,T_{0})$. The proof of the theorem is complete.

\subsection{Proof of Theorem \protect\ref{theo2}}

In this subsection we prove Theorem \ref{theo2}. We adapt a technique
exploited by \cite{LW} to derive blow-up type results for the parabolic
equation associated with the classical $p$-Laplace operator. We divide the
proof into two parts.

\noindent \textbf{Step 1 (}\textbf{Positive potential energies}\textbf{)}.
We first establish that there is a constant $\beta >\alpha $ such that%
\begin{equation}
\left\vert \left\Vert u\left( t\right) \right\Vert \right\vert
_{W_{0}^{s,p}\left( \overline{\Omega }\right) }\geq \beta   \label{1}
\end{equation}%
and%
\begin{equation}
\left\Vert u\left( t\right) \right\Vert _{L^{q}\left( \Omega \right) }\geq
C_{\ast }\beta   \label{2}
\end{equation}%
for all $t\geq 0$ (and for as long as the strong solution exists). First, we
notice that by definition (\ref{energy}) and the embedding (\ref{sob-emb}),
it holds%
\begin{align}
E\left( t\right) & \geq \frac{1}{p}\left\vert \left\Vert u\left( t\right)
\right\Vert \right\vert _{W_{0}^{s,p}\left( \overline{\Omega }\right) }^{p}-%
\frac{1}{q}C_{\ast }^{q}\left\vert \left\Vert u\left( t\right) \right\Vert
\right\vert _{W_{0}^{s,p}\left( \overline{\Omega }\right) }^{q}  \label{3} \\
& =\frac{1}{p}\tilde{x}^{p}-\frac{C_{\ast }^{q}}{q}\tilde{x}^{q}\overset{%
\text{def}}{=}h\left( \tilde{x}\right) ,  \notag
\end{align}%
where we have set $\tilde{x}:=\left\vert \left\Vert u\left( t\right)
\right\Vert \right\vert _{W_{0}^{s,p}\left( \overline{\Omega }\right) }$.
Clearly, the continuous function $h$ is increasing on $\left( 0,\alpha
\right) $ and decreasing on $\left( \alpha ,\infty \right) $ while $h\left(
\tilde{x}\right) \rightarrow -\infty $ as $\tilde{x}\rightarrow \infty $ and
$h\left( \alpha \right) =E_{0}$. Then, since $E\left( 0\right) <E_{0}$ it
immediately follows that one has a constant $\beta >\alpha $ such that $%
h\left( \beta \right) =E\left( 0\right) .$ On the other hand, setting $%
\tilde{x}_{0}=\left\vert \left\Vert u_{0}\right\Vert \right\vert
_{W_{0}^{s,p}\left( \overline{\Omega }\right) }$ then $h\left( \tilde{x}%
_{0}\right) \leq E\left( 0\right) =h\left( \beta \right) $ and $\tilde{x}%
_{0}\geq \beta $ on the account of (\ref{3}). In order to show (\ref{1}), we
proceed to prove it by contradiction. To this end, let us assume that $%
\left\vert \left\Vert u\left( t_{0}\right) \right\Vert \right\vert
_{W_{0}^{s,p}\left( \overline{\Omega }\right) }<\beta $ for some $t_{0}\in
\left( 0,T_{0}\right) $ on which the strong solution exists. By the
continuity of this norm we can choose $t_{0}>0$ such that $\left\vert
\left\Vert u\left( t_{0}\right) \right\Vert \right\vert _{W_{0}^{s,p}\left(
\overline{\Omega }\right) }>\alpha $. By (\ref{3}), we find that $E\left(
t_{0}\right) \geq h\left( \left\vert \left\Vert u\left( t_{0}\right)
\right\Vert \right\vert _{W_{0}^{s,p}\left( \overline{\Omega }\right)
}\right) >h\left( \beta \right) =E\left( 0\right) $ which contradicts the
fact that $E\left( t\right) \leq E\left( 0\right) $, for all $t\in \left(
0,T_{0}\right) ,$ on which the strong solution exists, with the latter
following easily by (\ref{en-ineq}). Hence, we have proved (\ref{1}). To
prove (\ref{2}), it remains to exploit (\ref{en-ineq}) once again together
with the definition of $E\left( t\right) $ and (\ref{1}) in order to see that%
\begin{equation*}
\frac{1}{q}\left\Vert u\left( t\right) \right\Vert _{L^{q}\left( \Omega
\right) }^{q}\geq \frac{1}{p}\left\vert \left\Vert u\left( t\right)
\right\Vert \right\vert _{W_{0}^{s,p}\left( \overline{\Omega }\right)
}^{p}-E\left( 0\right) \geq \frac{1}{p}\beta ^{p}-h\left( \beta \right) =%
\frac{C_{\ast }^{q}\beta ^{q}}{q}
\end{equation*}%
from which (\ref{2}) follows. Next, defining $H\left( t\right)
:=E_{0}-E\left( t\right) $, we have from \cite[Lemma 4]{LW} that
\begin{equation}
0<H\left( 0\right) \leq H\left( t\right) \leq \frac{1}{q}\left\Vert u\left(
t\right) \right\Vert _{L^{q}\left( \Omega \right) }^{q},  \label{4}
\end{equation}%
provided that $E\left( 0\right) <E_{0}$ and $\left\vert \left\Vert
u_{0}\right\Vert \right\vert _{W_{0}^{s,p}\left( \overline{\Omega }\right)
}>\alpha $, for as long as the strong solution exists.\newline

\noindent \textbf{Step 2 (}\textbf{Blow-up in }$L^{2}$\textit{-norm}\textbf{)%
}. As in \cite{LW} (and references therein), setting $G\left( t\right) :=%
\frac{1}{2}\left\Vert u\left( t\right) \right\Vert _{L^{2}\left( \Omega
\right) }^{2}$, we have%
\begin{align*}
G^{^{\prime }}\left( t\right) & =\int_{\Omega }u\left( t\right) \partial
_{t}u\left( t\right) dx=\int_{\Omega }\left\vert u\left( t\right)
\right\vert ^{q}dx-\left\vert \left\Vert u\left( t\right) \right\Vert
\right\vert _{W_{0}^{s,p}\left( \overline{\Omega} \right) }^{p} \\
& =\left( 1-p\right) \int_{\Omega }\left\vert u\left( t\right) \right\vert
^{q}dx-pH\left( t\right) -pE_{0},
\end{align*}%
owing to the definition of $H$. By Step 1, (\ref{1})-(\ref{2}), it is easy
to check that%
\begin{equation}
pE_{0}=\frac{\alpha ^{q}\left( q-p\right) }{\beta ^{q}q}C_{\ast }^{q}\beta
^{q}\leq \frac{\alpha ^{q}\left( q-p\right) }{\beta ^{q}}\left\Vert u\left(
t\right) \right\Vert _{L^{q}\left( \Omega \right) }^{q},  \label{5}
\end{equation}%
for as long as the strong solution exists. Setting $d=\left( 1-\alpha
^{q}/\beta ^{q}\right) \left( q-p\right) >0$, it follows that%
\begin{equation}
G^{^{\prime }}\left( t\right) \geq d\left\Vert u\left( t\right) \right\Vert
_{L^{q}\left( \Omega \right) }^{q}+pH\left( t\right) \geq 0.  \label{6}
\end{equation}%
On the other hand, by H\"{o}lder's inequality we observe that%
\begin{equation}
G^{\frac{q}{2}}\left( t\right) \leq L_{q,\Omega }\left\Vert u\left( t\right)
\right\Vert _{L^{q}\left( \Omega \right) }^{q},\text{ }L_{q,\Omega }:=\left(
\frac{1}{2}\right) ^{\frac{q}{2}}\left\vert \Omega \right\vert ^{\frac{q}{2}%
-1}.  \label{7}
\end{equation}%
Thus, combining (\ref{6})-(\ref{7}), we get $G^{^{\prime }}\left( t\right)
\geq \left( d/L_{q,\Omega }\right) G^{q/2}\left( t\right) $ and one can
directly integrate this inequality over $\left( 0,t\right) ,$ $t>0$. It
follows that%
\begin{equation*}
G^{\frac{q}{2}-1}\left( t\right) \geq \left( G^{1-q/2}\left( 0\right) -\frac{%
d}{d/L_{q,\Omega }}\left( \frac{q}{2}-1\right) t\right) ^{-1}
\end{equation*}%
which shows that $G\left( t\right) $ blows-up in finite time with a time $%
t\leq t_{\ast },$ given by (\ref{bl-time}). The proof is finished.

\section{Appendix}

We now prove Proposition \ref{subdif}, Lemma \ref{comp}, Lemma \ref{lem-9}
and Proposition \ref{en-in-prop}.

\begin{proof}[\bf Proof of Proposition \protect\ref{subdif}]
Let $f\in V^{\star }$ and $u\in V=W_{0}^{s,p}(\bOm)\cap L^{r}(\Omega )$.
We claim that $f=\partial _{V}\Phi (u)$ if and only if for every $v\in V$,
\begin{align} \label{int-sub}
&\langle f,v\rangle _{V^{\star },V}\\
=&\frac{C_{N,p,s}}{2}\int_{\mathbb{R}%
^{N}}\int_{\mathbb{R}^{N}}|u(x)-u(y)|^{p-2}\frac{(u(x)-u(y))(v(x)-v(y))}{%
|x-y|^{N+sp}}dxdy. \notag
\end{align}%
Indeed, let $f\in V^{\star }$ and $u\in V=W_{0}^{s,p}(\bOm )\cap
L^{r}(\Omega )$ be such that \eqref{int-sub} holds for every $v\in V$. Then %
\eqref{int-sub} holds with $v$ replaced by $v-u$. Using the following
well-known inequality
\begin{equation*}
\frac{|b|^{p}}{p}-\frac{|a|^{p}}{p}\geq |a|^{p-2}a(b-a),\text{ for any }%
a,b\in \mathbb{R},
\end{equation*}%
we get that for every $v\in V$,
\begin{align*}
&\Phi (v)-\Phi (u)\\
& =\frac{C_{N,p,s}}{2p}\int_{\mathbb{R}^{N}}\int_{\mathbb{R}%
^{N}}\frac{ |v(x)-v(y)|^{p}-|u(x)-u)(y)|^{p}}{|x-y|^{N+sp}}dxdy
\\
& \geq \frac{C_{N,p,s}}{2}\int_{\mathbb{R}^{N}}\int_{\mathbb{R}%
^{N}}|u(x)-u(y)|^{p-2}\frac{(u(x)-u(y))((v-u)(x)-(v-u)(y))}{|x-y|^{N+sp}}dxdy
\\
& =\langle f,v-u\rangle _{V^{\star },V}.
\end{align*}%
Hence, $f=\partial _{V}\Phi (u)$. Conversely, let $u\in V$ and set $%
f:=\partial _{V}\Phi (u)\in V^{\star }$. Then by definition, for every $v\in
V$, we have that
\begin{equation}
\Phi (v)-\Phi (u)\geq \langle f,v-u\rangle _{V^{\star },V}.  \label{subdif1}
\end{equation}%
Let $t\in \lbrack 0,1]$, $w\in V$ and set $v=tw+(1-t)u$ in \eqref{subdif1}.
Then
\begin{align}\label{subdif2}
&t\langle f,w-u\rangle _{V^{\star },V}\\
&\leq \frac{C_{N,p,s}}{2p}\int_{\mathbb{R%
}^{N}}\int_{\mathbb{R}^{N}}\frac{%
|(tw+(1-t)u)(x)-(tw+(1-t)u)(y)|^{p}-|u(x)-u(y)|^{p}}{|x-y|^{N+sp}}dxdy.\notag
\end{align}%
Using the Dominated Convergence Theorem, we get from \eqref{subdif2} that
\begin{align}
&\langle f,w-u\rangle _{V^{\star },V} \leq \lim_{t\downarrow 0}\frac{\Phi
(tw+(1-t)u)-\Phi (u)}{t}  \label{subdif3} \\
& =\frac{C_{N,p,s}}{2}\int_{\mathbb{R}^{N}}\int_{\mathbb{R}%
^{N}}|u(x)-u(y)|^{p-2}\frac{(u(x)-u(y))((w-u)(x)-(w-u)(y))}{|x-y|^{N+sp}}%
dxdy.  \notag
\end{align}%
Replacing $w$ by $u+w$ in \eqref{subdif3}, we get that for every $w\in V$,
\begin{align}\label{subdif4}
&\langle f,w\rangle _{V^{\star },V}\\
\leq& \frac{C_{N,p,s}}{2}\int_{\mathbb{R}%
^{N}}\int_{\mathbb{R}^{N}}|u(x)-u(y)|^{p-2}\frac{(u(x)-u(y))(w(x)-w(y))}{%
|x-y|^{N+sp}}dxdy.  \notag
\end{align}%
Since \eqref{subdif4} holds with $w$ replaced by $-w$, it follows that
\begin{equation*}
\langle f,w\rangle _{V^{\star },V}=\frac{C_{N,p,s}}{2}\int_{\mathbb{R}%
^{N}}\int_{\mathbb{R}^{N}}|u(x)-u(y)|^{p-2}\frac{(u(x)-u(y))(w(x)-w(y))}{%
|x-y|^{N+sp}}dxdy
\end{equation*}%
and we have shown \eqref{int-sub}. The proof of the claim is finished.
\end{proof}

\begin{proof}[\bf Proof of Lemma \protect\ref{comp}]
We prove the inequality by elementary analysis. Let the function $g:\mathbb{R%
}\times \mathbb{R}\rightarrow \mathbb{R}$ be given by%
\begin{equation}  \label{func-g}
g\left( z,t\right) =\left\vert z-t\right\vert ^{p-2}\left( z-t\right) \left(
\left\vert z\right\vert ^{r-2}z-\left\vert t\right\vert ^{r-2}t\right)
-C_{r,p}\left\vert \left\vert z\right\vert ^{\frac{r-2}{p}}z-\left\vert
t\right\vert ^{\frac{r-2}{p}}t\right\vert ^{p},
\end{equation}%
where we recall that
\begin{align*}
C_{r,p}=(r-1)\left(\frac{p}{r+p-2}\right)^p.
\end{align*}
Using the definition of $\mathcal{E}$, we first notice that (\ref%
{comp-energy}) is equivalent to showing that
\begin{equation}  \label{eq-GW}
g\left( z,t\right) \ge 0,\;\;\forall\;\left( z,t\right) \in \mathbb{R}^{2}.
\end{equation}%
Second, we mention that it is easy to verify that
\begin{align*}
g(z,t)=g(t,z),\;\;g(z,0)\ge 0,\; g(0,t)\ge 0\;\mbox{ and }\; g(z,t)=g(-z,-t).
\end{align*}
Therefore, without any restriction, we may assume that $z\ge t$ and hence,
we have that
\begin{equation*}
g\left( z,t\right) =(z-t)^{p-1} \left( \left\vert z\right\vert
^{r-2}z-\left\vert t\right\vert ^{r-2}t\right) -C_{r,p}\left\vert \left\vert
z\right\vert ^{\frac{r-2}{p}}z-\left\vert t\right\vert ^{\frac{r-2}{p}%
}t\right\vert ^{p},
\end{equation*}%
A simple calculation shows that
\begin{align}  \label{F1}
\frac{p}{r+p-2}\left[\left\vert z\right\vert ^{\frac{r-2}{p}}z-\left\vert
t\right\vert ^{\frac{r-2}{p}}t\right]=\int_t^z|\tau|^{\frac{r-2}{p}}\;d\tau.
\end{align}
Since the function $\varphi:\; \mathbb{R}\to\mathbb{R}$ given by $%
\varphi(\tau)=|\tau|^{p}$ ($p>1$) is convex, then using the well-known
Jensen inequality, it follows from \eqref{F1} that
\begin{align*}
C_{r,p}\left\vert \left\vert z\right\vert ^{\frac{r-2}{p}}z-\left\vert
t\right\vert ^{\frac{r-2}{p}}t\right\vert ^{p}&=(r-1)\left|\frac{p}{r+p-2}%
\left[\left\vert z\right\vert ^{\frac{r-2}{p}}z-\left\vert t\right\vert ^{%
\frac{r-2}{p}}t\right]\right|^p \\
&=(r-1)\left|\int_t^z|\tau|^{\frac{r-2}{p}}\;d\tau\right|^p \\
&=(r-1)(z-t)^p\left|\int_t^z|\tau|^{\frac{r-2}{p}}\;\frac{d\tau}{z-t}%
\right|^p \\
&\le (r-1)(z-t)^p\int_t^z|\tau|^{r-2}\;\frac{d\tau}{z-t} \\
&=(r-1)(z-t)^{p-1}\int_t^z|\tau|^{r-2}\;d\tau \\
&=(z-t)^{p-1} \left(\left\vert z\right\vert ^{r-2}z-\left\vert t\right\vert
^{r-2}t\right).
\end{align*}
We have shown \eqref{eq-GW} and this completes the proof of lemma.
\end{proof}

\begin{proof}[\bf Proof of Lemma \protect\ref{lem-9}]
Let $u\in D(\partial _{V}\Phi )=V$ and $\mu >0$. It follows from Proposition %
\ref{more} that $J_{\mu }^{\phi }u\in V=D(\partial _{V}\Phi )$. Therefore,
\begin{equation*}
\partial _{H}\phi _{\mu }(u)=\frac{u-J_{\mu }^{\phi }u}{\mu }\in V,\;\;\text{%
for all}\;u\in V.
\end{equation*}

Next, let $w\in D(\partial _{V}\Phi )\cap D(\partial _{H}\phi )=V$ be such
that $\partial _{H}\phi (w)\in V$. Since $\partial _{H}\phi
(w)=|w|^{r-2}w\in W_{0}^{s,p}(\bOm )$, it follows from Lemma \ref{comp}
that there exists a constant $\beta =\beta (r,p,s)>0$ such that
\begin{align}
& \frac{\beta C_{N,p,s}}{2}\int_{\mathbb{R}^{N}}\int_{\mathbb{R}^{N}}\frac{%
\left\vert |w(x)|^{\frac{r-2}{p}}w(x)-|w(y)|^{\frac{r-2}{p}}w(y)\right\vert
^{p}}{|x-y|^{N+sp}}dxdy  \label{EE1} \\
\leq & \frac{C_{N,p,s}}{2}\int_{\mathbb{R}^{N}}\int_{\mathbb{R}%
^{N}}|w(x)-w(y)|^{p-2}\frac{(w(x)-w(y))(|w(x)|^{r-2}w(x)-|w(y)|^{r-2}w(y))}{%
|x-y|^{N+sp}}dxdy  \notag \\
=& \int_{\Omega }(-\Delta )_{p}^{s}w|w|^{r-2}wdx=\int_{\mathbb{R}%
^{N}}(-\Delta )_{p}^{s}w|w|^{r-2}wdx=\langle \partial _{V}\Phi (w),\partial
_{H}\phi (w)\rangle _{V^{\star },V}.  \notag
\end{align}%
Now, let $v_{\mu }:=|J_{\mu }^{\phi }u|^{\frac{r-2}{p}}J_{\mu }^{\phi }u$.
Note that $v_\mu=0$ on $\mathbb{R}^N\setminus\Omega$ and using that $%
\partial _{H}\phi (J_{\mu }^{\phi }u)=\partial _{H}\phi _{\mu }(u)\in V$, we
get that
\begin{align*}
\int_{\mathbb{R}^N}|v_\mu|^p\;dx&=\int_{\Omega}|v_\mu|^p\;dx=\int_{%
\Omega}|J_{\mu }^{\phi }u|^{r-2+p}\;dx\le \int_{\Omega}|u|^{r-2+p}\;dx \\
&\le \int_{\Omega}|u|^{r-1}|u|^{p-1}\;dx\le
\left(\int_{\Omega}|u|^{p(r-1)}\;dx\right)^{\frac
1p}\left(\int_{\Omega}|u|^p\;dx\right)^{\frac{p-1}{p}}<\infty.
\end{align*}
Since $J_{\mu }^{\phi }u\in D(\partial _{V}\Phi )=V$ and $\partial _{H}\phi
(J_{\mu }^{\phi }u)=\partial _{H}\phi _{\mu }(u)\in V$, \eqref{EE1} allows
us to deduce that $v_{\mu }\in W_{0}^{s,p}(\bOm )$. We have shown the
first claim (\ref{EE0}).

Finally, assume that $u\in D(\partial _{V}\Phi ^{\sigma })=V_{\sigma }$.
Then by Propositions \ref{more} and \ref{prop-24}, there holds $J_{\mu
}^{\phi }u\in V_{\sigma }$. Hence, for all $u\in D(\partial _{V}\Phi
^{\sigma })=V_{\sigma }$ and $g\in \partial _{V}\Phi ^{\sigma }(u)$, we have
the estimate%
\begin{align}
\langle g,\partial _{H}\phi _{\lambda }(u)\rangle _{V^{\star },V}\geq &
\lambda ^{-1}\left( \Phi ^{\sigma }(u)-\Phi ^{\sigma }(J_{\lambda }^{\phi
}\left( u\right) )\right)  \label{EE2} \\
=& \lambda ^{-1}\left[ \Phi (u)-\Phi (J_{\lambda }^{\phi }\left( u\right) )%
\right]  \notag \\
\geq & \lambda ^{-1}\langle \partial _{V}\Phi (J_{\lambda }^{\phi }\left(
u\right) ),u-J_{\lambda }^{\phi }\left( u\right) \rangle _{V^{\star },V}
\notag \\
=& \lambda ^{-1}\langle \partial _{V}\Phi (J_{\lambda }^{\phi }\left(
u\right) ),\partial _{H}\phi _{\lambda }(u)\rangle _{V^{\star },V}.  \notag
\end{align}%
Combining \eqref{EE1} together with \eqref{EE2}, we easily obtain \eqref{EE0}%
. This finishes the proof of lemma.
\end{proof}

\begin{proof}[\bf Proof of Proposition \protect\ref{en-in-prop}]
As in Step 4 of the proof of Theorem \ref{main-theo}, we can pick a
sufficiently smooth sequence of initial data $u_{0,n}\in D(\Phi )\cap V$
such that $u_{0,n}\rightarrow u_{0}$ strongly in $V=W_{0}^{s,p}(\bOm )\cap
L^{r}(\Omega )$ as $n\to\infty$. Then we consider again the approximate problem (\ref%
{par-pro-n}) on $\left( 0,T_{0}\right) $ (of course, now $f_{n}\equiv 0$)
which we test it again with $\partial _{t}u_{n}\in L^{2}\left(
(0,T_{0});L^{2}\left( \Omega \right) \right) $. We note that every smooth
solution of the approximate problem (\ref{par-pro-n}) does possess such
regularity. We obtain%
\begin{equation}
\frac{d}{dt}E_{n}\left( t\right) +\left\Vert \partial _{t}u_{n}\left(
t\right) \right\Vert _{L^{2}\left( \Omega \right) }^{2}=0,  \label{en-smooth}
\end{equation}%
for all $t\in \left( 0,T_{0}\right) $, and where we have set%
\begin{equation*}
E_{n}\left( t\right) :=\frac{C_{N,p,s}}{2p}\int_{\mathbb{R}^{N}}\int_{%
\mathbb{R}^{N}}\frac{|u_{n}(x,t)-u_{n}(y,t)|^{p}}{|x-y|^{N+sp}}dxdy-\frac{1}{%
q}\int_{\Omega }|u_{n}\left( x,t\right) |^{q}dx.
\end{equation*}%
Integrating (\ref{en-smooth}) over the interval $\left( 0,t\right) $ allows
us to deduce%
\begin{equation*}
E_{n}\left( t\right) \leq E_{n}\left( 0\right)
\end{equation*}%
for all $t\in \left( 0,T_{0}\right) $. We can now easily conclude the proof
of Proposition \ref{en-in-prop} exploiting the foregoing inequality. Indeed,
recalling that $r>\frac{N(q-p)}{sp}$ with $q>p$, we see that $u_{n}\left(
t\right) \rightarrow u\left( t\right) $ strongly in $L^{q}\left( \Omega
\right) $, a.e. for $t\in \left( 0,T_{0}\right) $, owing to (\ref{eq60}).
Passing to the limit as $n\rightarrow \infty $, we first have $E_{n}\left(
0\right) \rightarrow E\left( 0\right) $ and then also%
\begin{equation*}
\int_{\Omega }|u_{n}\left( x,t\right) |^{q}dx\rightarrow \int_{\Omega
}|u\left( x,t\right) |^{q}dx\text{ a.e. for }t\in \left( 0,T_{0}\right) .
\end{equation*}%
This basic fact together with the weak lower-semicontinuity of the $%
W_{0}^{s,p}(\bOm)$-norm entails that $E\left( t\right) \leq \lim
\inf_{n\rightarrow \infty }E_{n}\left( t\right) $ and this concludes the
proof of (\ref{en-ineq}).
\end{proof}

\end{document}